\documentclass[12pt]{article}
\usepackage{amsfonts}
\usepackage{amssymb}
\usepackage{amssymb,amsfonts}
\usepackage{epsfig}
\usepackage{CJK}
\usepackage{amsthm}
\usepackage{amsmath}
\usepackage{graphicx}

 \newtheorem{thm}{Theorem}[section]
 
 \newtheorem{lem}[thm]{Lemma}
 
 \theoremstyle{definition}
 \newtheorem{defn}[thm]{Definition}
 \newtheorem{rem}[thm]{Remark}
 \numberwithin{equation}{section}

\newtheorem{Lemma A.1}{Lemma A.1}
\theoremstyle{definition}

\theoremstyle{remark}

\newcommand{\Real}{\mathbb R}
\newcommand{\eps}{\epsilon}

\begin{document}
\title{Global Regularity to the Navier-Stokes Equations for A Class of Large Initial Data}
\author{Yukang Chen\footnotemark[1]\and Bin Han
\footnotemark[2]\and
Zhen Lei\footnotemark[3]}
\renewcommand{\thefootnote}{\fnsymbol{footnote}}

\footnotetext[1]{School of Mathematical Sciences, Fudan University,
Shanghai,  200433,  China.}

\footnotetext[2]{ Corresponding author and Email: hanbinxy@163.com. School of Mathematical Sciences, Fudan University,
Shanghai,  200433,  China. }

\footnotetext[3]{School of Mathematical Sciences and Shanghai Center for Mathematical Sciences, Fudan University,
Shanghai,  200433,  China. }

\date{}
\maketitle

\begin{abstract}
Consider the generalized Navier-Stokes equations on $ \Real^n$:
$$\partial_tu+u\cdot\nabla u + D^s u+\nabla P=0,\quad \mathrm{div}\;u=0.$$
For some appropriate number $s$, we prove that the Cauchy problem with initial data of the form
\begin{equation}\nonumber
u_0^\epsilon(x) = (v_0^h(x_\epsilon), \epsilon^{-1}v_0^n(x_\epsilon))^T,\quad x_\epsilon = (x_{h}, \epsilon x_n)^T,
\end{equation}
is globally well-posed for all small $\epsilon > 0$, provided that the initial velocity profile $v_0$ is analytic in $x_n$ and certain norm of $v_0$ is sufficiently small but independent of $\epsilon$. In particular, for $n\geq4$, our result is applicable to the n-dimensional classical Navier-Stokes equations.
\end{abstract}

\section{Introduction}
The Cauchy problem of the  incompressible Navier-Stokes equations on $\Real^n$ is described by the following system
\begin{equation}\label{e1.1}
\left\{
 \begin{array}{rlll}
 &\partial_tu+u\cdot\nabla u-\Delta u+\nabla P=0,
\ \ &\ x\in\Real^n,\ t>0,\\
  &\hbox{div}\, u=0, \ \ &\ x\in\Real^n,\ t>0,\\
  &u(0)=u_0,  \ \ &\ x\in \Real^n,
   \end{array}
  \right.
\end{equation}
where $u$  represents the velocity field and $P$ is the scalar pressure.

First of all, let us recall some known results on the small-data global regularity of the Navier-Stokes equations on $\Real^3$. In the seminal paper \cite{Le},  Leray proved that the 3D incompressible Navier-Stokes equations are globally well-posed if the initial data  $u_0$ is such that $\|u_0\|_{L^2}\|\nabla u_0\|_{L^2}$ is small enough. This quantity  is invariant under the natural scaling of the Navier-Stokes equations. Later on, many authors studied different scaling invariant spaces in which Navier-Stokes equations are well-posed at least for small initial data, which include but are not limited to
$$
\dot{H}^{\frac{1}{2}}(\Real^3)\hookrightarrow L^3(\Real^3)\hookrightarrow
\dot{B}^{-1+\frac{3}{p}}_{p,\infty}(\Real^3)\hookrightarrow BMO^{-1}(\Real^3),
$$
where $3<p<\infty$. The space $BMO^{-1}(\Real^3)$ is known to be the largest scaling invariant space
so that the Navier-Stokes equations (\ref{e1.1}) are globally well-posed under small initial data. The readers are referred to \cite{FK,Ka,CMP,KT} as references. We also mention that the work of Lei and Lin \cite{LL} was the first to quantify the smallness of the initial data to be 1 by introducing a new space $\mathcal{X}^{-1}$.

We remark that the norm in the above scaling invariant spaces are always
greater than the norm in the Besov space $\dot{B}^{-1}_{\infty,\infty}$ defined by
$$\|u\|_{\dot{B}^{-1}_{\infty,\infty}}\overset{\mathrm{def}}{=}\sup\limits_{t>0}t^{\frac{1}{2}}\|e^{t\Delta}u_0\|_{L^\infty}.$$
Bourgain and Pavlovic in \cite{Bour} showed that the cauchy problem of the 3D Navier-Stokes equations is ill-posed in the sense of norm inflation. Partially because of the result of Bourgain and Pavlovic, data with a large $\dot{B}^{-1}_{\infty, \infty}$ are usually called large data to the Navier-Stokes equations (for instance, see \cite{CG,PZ}).

 Towards this line of research, a well-oiled case is the family of initial data  which is slowly varying in vertical variable.
The initial velocity field $u_0^\eps$  is of the form
\begin{equation}\label{e1.2}
u_0^\eps(x) = ( v_0^{h}(x_\eps),\eps^{-1}v_0^{3}(x_\eps))^T,\quad x_\eps = (x_h, \epsilon x_3)^T,
\end{equation}
which allows slowly varying in the vertical variable $x_3$ when $\eps > 0$ is a small parameter.
This family of initial data are very interesting (as has been pointed out by \mbox{V.$\check{\text{S}}$ver$\acute{\text{a}}$k}, see the acknowledgement in \cite{CGP}) and considered by Chemin, Gallagher and Paicu in \cite{CGP}. They proved the global regularity of solutions to the Navier-Stokes equations when $v_0$ is analytic in $x_3$ and periodic in $x_h$, and certain norm of $v_0$ is sufficiently small but independent of $\eps > 0$. More precisely, they proved the following Theorem:
 \begin{thm}[Chemin-Gallagher-Paicu, Ann. Math. 2011]\label{t1.1}
 Let $\alpha$ be a positive number. There are two positive numbers $\eps_0$ and $\eta$ such that for any divergence free
 vector field $v_0$ satisfying
 $$\|e^{\alpha D_3 }v_0\|_{{B}^{\frac{7}{2}}_{2, 1}}\leq \eta,$$
 then,  for any positive $\eps$ smaller than $\eps_0$,  the initial data \eqref{e1.2}
generates a global smooth solution to (\ref{e1.1}) on $\mathbb{T}^2\times \Real.$
\end{thm}

The notation ${B}^{\frac{7}{2}}_{2, 1}$ in the above Theorem denotes the usual inhomogeneous Besov space. The significance of the result lies in that the global regularity of the 3D incompressible Navier-Stokes equations in \cite{CGP} only requires \textit{very little smallness} imposed on the initial data. It is clear that the $\dot{B}^{-1}_{\infty, \infty}$ norm of $u_0^\eps$ can tend to infinity as $\epsilon \to 0$. Let us first focus on the periodic constraint imposed on the initial data in Theorem \ref{t1.1}.

As has been pointed out by Chemin, Gallagher and Paicu, the reason why the horizontal variable of the initial data in \cite{CGP} is restricted to a torus is to be able to deal with very low horizontal frequencies. In the proof of Theorem \ref{t1.1} in \cite{CGP},
functions with zero horizontal average are treated differently to the others, and it
is important that no small horizontal frequencies appear other than zero. Later on, many efforts are made towards removing the periodic constraint of $v_0$ on the horizontal variables.  For instance, see \cite{CG,Ha,GHZ,PZ} and so on.  We will review those results a little bit later.

In this paper, we consider the Cauchy problem of the following generalized Navier-Stokes equations on $\Real^n$:
\begin{equation}\label{e1.3}
\left\{
 \begin{array}{rlll}
 &\partial_tu+u\cdot\nabla u+D^su+\nabla P=0,
\ \ &\ x\in\Real^n,\ t>0,\\
  &\hbox{div}\, u=0, \ \ &\ x\in\Real^n,\ t>0,\\
  &u(0)=u^\eps_0,  \ \ &\ x\in \Real^n,
   \end{array}
  \right.
\end{equation}
where $D=\sqrt{-\Delta}$.
The initial velocity field $u_0^\eps$  is of the form
\begin{equation}\label{e1.4}
u_0^\eps(x) = ( v_0^{h}(x_\eps),\eps^{-1}v_0^{n}(x_\eps))^T,\quad x_\eps = (x_h, \epsilon x_n)^T.
\end{equation}
The horizontal variable $x_h=(x_1,x_2,\cdots,x_{n-1})$.

Our main result is the following theorem which  generalizes the theorem of Chemin, Gallagher and Paicu  to the whole space for the generalized Navier-Stokes equations with some appropriate number $s$. Definition of notations will be given in Section 2.
\begin{thm}\label{t1.2}
Let $\alpha$, $\epsilon_0$, $p$ and $s$ be four positive constants and $(p,s)$ satisfy $1\leq p < n-1, 1 \leq s < \min(n-1,\frac{2(n-1)}{p})$. There exists a positive constant $\eta$
such that  for any $0 < \eps < \eps_0$ and any divergence free vector field $v_0$ satisfying
 \begin{equation}\label{e1.5}
\|e^{\alpha D_n }v_0\|_{\dot{B}^{\frac{n-1}{p}-s,\frac{1}{2}}_{p,1}\cap \dot{B}^{\frac{n-1}{p}+1-s,\frac{1}{2}}_{p,1}}\leq \eta,
\end{equation}
then the generalized Navier-Stokes equations (\ref{e1.3}) with  initial data
\eqref{e1.4}
generate a global smooth solution on $\Real^n$.
\end{thm}
\noindent
\begin{rem}

1) When $n\geq4$, one can find that the classical Navier-Stokes equations on $\Real^n$  satisfy the assumption $1\leq s=2 < \min(n-1,\frac{2(n-1)}{p})$ for $1\leq p <n-1$.  Then according to Theorem \ref{t1.2}, the n-dimensional incompressible Navier-Stokes equations with initial data (\ref{e1.4}) have a global smooth solution in the whole space case.\\
2) In the case $n=3$, we require that $1\leq s<2$. The main obstacle when $s=2$ is that we can not get the product law in $\dot{B}^{\frac{2}{p}-2,\frac{1}{2}}_{p,1}(\Real^3).$  This anisotropic Besov space is induced by the $a \ priori$ estimate of  anisotropic pressure $(\nabla_hq, \eps^2\partial_3q)$ (see equation (\ref{e1.6}) and Step 5 for details). From this point of view, the global well-posedness of 3D incompressible  Navier-Stokes equations with initial data (\ref{e1.4}) on $\mathbb{R}^3$ is still unclear, even though the higher dimensional cases are settled down.\\
3) In the present paper, we establish the global solution in the $L^p$-type Besov space, in which the bilinear estimate of the solution can not be derived by  the classical $L^2$ energy method. Particularly, one can obtain the $L^1$ estimate in time of the solution by introducing the new quantity
 $$\int_0^t \|v^n\|_{\dot{B}_{p,1}^{\frac{n-1}{p},\frac{1}{2}}}\|\partial_nv^h\|_{\dot{B}_{p,1}^{\frac{n-1}{p}+1-s,\frac{1}{2}}}d\tau$$
in the $a\ priori$ estimate. We also mention that in \cite{CGP,PZ,PZ2}, authors did not get the  $L^1$-time estimate of the solution.
\end{rem}

Now we mention that many authors make efforts to remove the periodic restriction on horizontal variable.
J. Chemin and I. Gallagher  considered the well-prepared initial data in \cite{CG}. They proved the global well-posedness of (\ref{e1.1}) when $u_0^\eps$ is of the form
$$u_0^\eps{=}(v^h_0+\eps w^h_0, w_0^3)(x_h,\eps x_3).$$
Later, G. Gui, J. Huang, and P. Zhang in  \cite{GHZ} generalized this result to the density dependent Navier-Stokes
equations with the same initial velocity.  Recently, B. Han in \cite{Ha} considered the global regularity of (\ref{e1.1}) if $u_0^\eps$ satisfies the form of
$${u}^\eps_0(x){=}\eps^\delta(v_0^{h}( x_h,\eps x_3),\eps^{-1}v_0^{3}( x_h,\eps x_3))$$
for any $0<\delta<1$,
then $u^\eps_0$ generates a global solution of (\ref{e1.1}) on $\Real^3.$ The case $\delta=\frac{1}{2}$ and $\delta\in(0,\frac{1}{2})$ were considered by
M. Paicu and Z. Zhang in \cite{PZ,PZ2} if  the initial data is allowed to be in Gevery class. All of the initial data is large in $\dot{B}^{-1}_{\infty,\infty}$, but still generates a global
solution.

{\bf{Main ideas of the Proof.}} We will prove our main result by constructing the bilinear  estimate  (independent of $\eps$). Our strategy can be stated as follows.

{\bf{Step 1.}} Rescaled system and simplification.

As in \cite{CGP}, we define
$$u^\eps(t,x)=( v^h(t,x_\eps), \eps^{-1}v^n(t,x_\eps))^T,\quad P^\eps(t,x)=q(t,x_\eps).$$
Denote
$$\Delta_\eps=\Delta_h+\eps^2\partial_n^2, \  \Delta_h=\partial_1^2+\cdots+\partial_{n-1}^2,\ D_\eps=\sqrt{-\Delta_\eps}.$$
Using the Navier-Stokes equations \eqref{e1.1}, it is easy to derive the equations governing the rescaled variables  $v$ and $q$ (they are still depending on $\eps$):
\begin{equation}\label{e1.6}
\left\{
 \begin{array}{rlll}
 &\partial_tv^h+ v\cdot\nabla v^h+D_\eps^s v^h+\nabla_h q=0,\\
 &\partial_tv^n+v\cdot\nabla v^n+D_\eps^s v^n+\eps^2\partial_n q=0,\\
 & \hbox{div}\,v=0,\quad v(0)=v_0(x),
   \end{array}
  \right.
\end{equation}
where $v^h=(v^1,\cdots,v^{n-1})$.
The rescaled pressure $q$ can be recovered  by the divergence
free condition as
$$-\Delta_\eps q= \sum\limits_{i,j}\partial_i\partial_j(v^iv^j).$$
The global regularity of solutions to system \eqref{e1.4} for small initial data $v_0$ will be presented in Section 3 and 4 for any positive $\eps$. But to best illustrate our ideas, let us here focus on the case of $\eps = 0$.
Formally, by taking $\eps=0$ in system \eqref{e1.4},
we have the following limiting system:
 \begin{equation}\label{elimit}
\left\{
 \begin{array}{rlll}
 &\partial_tv^h+ v\cdot\nabla v^h+D_h^s v^h+\nabla_h q=0,\\
 &\partial_tv^n+v\cdot\nabla v^n+D_h^s v^n=0,\\
  &\hbox{div}\,v=0,\quad v(0)=v_0(x),
   \end{array}
  \right.
\end{equation}
where $D_h=\sqrt{-\Delta_h}$.
The pressure $q$ in (\ref{elimit}) is given by
$$-\Delta_h q= \sum\limits_{i,j}\partial_i\partial_j(v^iv^j).$$

{\bf{Step 2.}} Set-up of the $a\ priori$ estimate.

 Observing that in the rescaled system (\ref{elimit}), the  viscosity is absent in the vertical direction. To make the full use of smoothing effect from operator $\partial_t +D_h^s$, particularly in low  frequency parts, we will
apply the tool of anisotropic homogeneous Besov spaces.
The goal is to derive certain $a\ priori$ estimate of the form:
 $$\Psi(t)\lesssim \Psi(0)+\Psi(t)^2.$$
Note that pressure term doesn't explicitly appear in the equation of $v^n$ of the limiting system (\ref{elimit}), which makes the estimate for $v^n$ easier.
So here let us just focus on the equation of $v^h$.  Naturally, we define
 $$\Psi(t)= \|v^h\|_{\widetilde{L}^\infty_t(\dot{B}_{p,1}^{\frac{n-1}{p}+1-s,\frac{1}{2}})}+\|v^h\|_{L^1_t(\dot{B}_{p,1}^{\frac{n-1}{p}+1,\frac{1}{2}})}+\cdots.$$
 At this step,  we assume that the initial data $v_0^h$ belongs to $\dot{B}_{p,1}^{\frac{n-1}{p}+1-s,\frac{1}{2}}$.
 This ensures that  $\Psi(t)$
is a critical quantity with respect to the natural scaling of the generalized
Navier-Stokes equations.
By  Duhamel's principle, we can write the integral equation of $v^h$ by
$$v^h = e^{-tD_h^s}v_0^h - \int_0^te^{-(t - \tau)D_h^s}(v^h\cdot\nabla_h v^h + v^n\partial_nv^h + \nabla_hq)d\tau.$$
According to the estimates of heat equation,
one can formally has
$$\Psi(t)\lesssim \Psi(0)+\|v^h\cdot\nabla_h v^h\|_{L^1_t(\dot{B}_{p,1}^{\frac{n-1}{p}+1-s,\frac{1}{2}})} +\| v^n\partial_nv^h\|_{L^1_t(\dot{B}_{p,1}^{\frac{n-1}{p}+1-s,\frac{1}{2}})} + \|\nabla_hq\|_{L^1_t(\dot{B}_{p,1}^{\frac{n-1}{p}+1-s,\frac{1}{2}})}.$$
It will be shown that
$$\aligned
\|v^h\cdot\nabla_h v^h\|_{L^1_t(\dot{B}_{p,1}^{\frac{n-1}{p}+1-s,\frac{1}{2}})}\lesssim \|v^h\|_{\widetilde{L}^\infty_t(\dot{B}_{p,1}^{\frac{n-1}{p}+1-s,\frac{1}{2}})}\|\nabla_h v^h\|_{L^1_t(\dot{B}_{p,1}^{\frac{n-1}{p},\frac{1}{2}})} \lesssim \Psi(t)^2.
\endaligned$$

{\bf{Step 3.}} Derivative loss: input the estimate of $\partial_nv^h$.

For the quantity $\Psi(t)$, we need to prove the  bilinear estimate in the following form:
$$\Psi(t)\lesssim \Psi(0)+\Psi(t)^2+ \|v^n\partial_nv^h\|_{ L^1_t(\dot{B}_{p,1}^{\frac{n-1}{p}+1-s,\frac{1}{2}})}+\cdots.$$
 Certainly, there is $\partial_n$-derivative loss! By the product law in anisotropic Besov spaces (Lemma \ref{p2.5}), the strategy  to bound the last term is
 $$\|v^n\partial_nv^h\|_{L^1_t(\dot{B}_{p,1}^{\frac{n-1}{p}+1-s,\frac{1}{2}})}\lesssim  \int_0^t \|v^n\|_{\dot{B}_{p,1}^{\frac{n-1}{p},\frac{1}{2}}}\|\partial_nv^h\|_{\dot{B}_{p,1}^{\frac{n-1}{p}+1-s,\frac{1}{2}}}d\tau,$$
and then we should add the new quantity
 $$\int_0^t \|v^n\|_{\dot{B}_{p,1}^{\frac{n-1}{p},\frac{1}{2}}}\|\partial_nv^h\|_{\dot{B}_{p,1}^{\frac{n-1}{p}+1-s,\frac{1}{2}}}d\tau$$
 in the definition of $\Psi(t).$
 We find that $\|\partial_nv^h\|_{\dot{B}_{p,1}^{\frac{n-1}{p}+1-s,\frac{1}{2}}}$ is the hardest term to estimate. Since we note  that  by  Duhamel's principle,
  $$\partial_nv^h = e^{-tD_h^s}\partial_nv_0^h - \int_0^te^{-(t - \tau)D_h^s}\partial_n(v^n\partial_nv^h)d\tau+\cdots.$$
 There is still  $\partial_n$-derivative loss in the $a\ priori$ estimates!

{\bf{Step 4.}} Recover the derivative loss: analyticity in $x_n$.

Motivated by Chemin-Gallagher-Paicu \cite{CGP}, we add an exponential
weight $e^{\Phi(t,  D_n )}$ with
$$\Phi(t, |\xi_n|)=(\alpha-\lambda\theta(t))|\xi_n|.$$
Here $\theta(t)$ is defined by
$$\theta(t)=\int_0^t\|v_\Phi^n\|_{\dot{B}_{p,1}^{\frac{n-1}{p},\frac{1}{2}}}d\tau,$$
which will be shown to be small to ensure that $\Phi(t, |\xi_n|)$ satisfies the subadditivity. Denoting $f_\Phi=e^{\Phi(t,  D_n )}f$, we then have
 $$\partial_nv_\Phi^h = e^{-tD_h^s}e^{\Phi(t,  D_n )}\partial_nv_0^h - \int_0^te^{-(t - \tau)D_h^s}e^{-\lambda\int_\tau^t\dot{\theta}(t')dt' D_n }\partial_n(v^n\partial_nv^h)_\Phi d\tau+\cdots.$$
 Hence, we can recover the derivative loss by
 $$\aligned
 &\int_0^t \|v_\Phi^n\|_{\dot{B}_{p,1}^{\frac{n-1}{p},\frac{1}{2}}} \|\partial_nv_\Phi^h\|_{\dot{B}_{p,1}^{\frac{n-1}{p}+1-s,\frac{1}{2}}}d\tau\\
 &\lesssim \sum\limits_{k,j}2^{k(\frac{n-1}{p}+1-s)}2^{\frac{1}{2}j}\int_0^t\dot{ \theta}(\tau)\int_0^{\tau}e^{-c\lambda\int_{t'}^\tau\dot{\theta}(t'')dt''2^j}2^j\|\Delta_{k,j}(v^n\partial_nv^h)_\Phi\|_{L^p_h(L^2_v)}
 dt' d\tau+\cdots\\
  &\lesssim \sum\limits_{k,j}2^{k(\frac{n-1}{p}+1-s)}2^{\frac{1}{2}j}\int_0^t\int_{t'}^{t}e^{-c\lambda\int_{t'}^\tau\dot{\theta}(t'')dt''2^j}2^j\dot{ \theta}(\tau)d\tau\|\Delta_{k,j}(v^n\partial_nv^h)_\Phi\|_{L^p_h(L^2_v)}
 dt' +\cdots\\
 &\lesssim \frac{1}{\lambda}\int_0^t\|v^n_\Phi\|_{\dot{B}_{p,1}^{\frac{n-1}{p},\frac{1}{2}}}\|\partial_nv^h_\Phi\|_{\dot{B}_{p,1}^{\frac{n-1}{p}+1-s,\frac{1}{2}}}
 dt' +\cdots.
 \endaligned$$
In this way, the losing derivative term  can be absorbed by the left hand side of the above inequality by choosing $\lambda$ sufficient large.

{\bf{Step 5.}} The estimate of  the pressure term $\nabla_hq$.\\
To estimate the pressure term, we write
$$\nabla_hq = -2\nabla_h(-\Delta_h)^{-1}(v^n\partial_n\mathrm{div}_hv^h)+2\nabla_h(-\Delta_h)^{-1}(\mathrm{div}_hv^h\mathrm{div}_hv^h)+\cdots$$
The  $ L^1_t(\dot{B}_{p,1}^{\frac{n-1}{p}+1-s,\frac{1}{2}})$ norm of the first term in $\nabla_hq$  can be estimated by
$$\|\nabla_h(-\Delta_h)^{-1}(v^n\partial_n\mathrm{div}_hv^h)\|_{L^1_t(\dot{B}_{p,1}^{\frac{n-1}{p}+1-s,\frac{1}{2}})}\lesssim
\|(v^n\partial_n\mathrm{div}_hv^h)\|_{L^1_t(\dot{B}_{p,1}^{\frac{n-1}{p}-s,\frac{1}{2}})}.$$
Here we should point out that when $s=2,n=3$, system (\ref{e1.3}) is nothing but the 3D incompressible Navier-Stokes equations.
In this case we have to deal with $\|fg\|_{\dot{B}_{p,1}^{\frac{2}{p}-2,\frac{1}{2}}}$ type estimate.
Unfortunately, the product law in $\dot{B}_{p,1}^{\frac{2}{p}-2,\frac{1}{2}}$ is  hard to obtain since we can not control the low horizontal frequency part.

{\bf{Step 6.}} Estimate of $\theta(t).$

In this step, we want to prove that for any time $t$,  $\theta(t)$ is a small quantity. This ensures that the phase function $\Phi$ satisfies the subadditivity property.
We will go to derive a stronger estimate for

$$\aligned
Y(t)=&\|{v}^n_{\Phi}\|_{\widetilde{L}_t^\infty(\dot{B}_{p,1}^{\frac{n-1}{p}-s,\frac{1}{2}})}
+\|{v}^n_{\Phi}\|_{{L}_t^1(\dot{B}_{p,1}^{\frac{n-1}{p},\frac{1}{2}})}.
\endaligned$$
However, when $\eps>0$, we can not get the closed estimate for $Y(t)$.
Our strategy is to add
an extra term $\epsilon v^h$ under the same norm which is hidden in the pressure term $\eps^2\partial_nq.$ See section 4 for details.

{\bf{Step 7.}} Estimate of $v^n_\Phi.$

If this is done,  we could get a closed \textit{a priori} estimate (see Lemma \ref{l3.2}) and finish the proof of Theorem \ref{t1.2}.  Observing that the nonlinear term $v^n\partial_nv^n$ can be rewritten as $-v^n\mathrm{div}_h{v}^h$ due to divergence free condition. Hence, in the limiting system (\ref{elimit}), there is no loss of derivative in vertical direction on $v^n$. Thus the estimate on $v^n$ is  much easier than $v^h.$

There are also some other type of large initial data so that the Navier-Stokes equations are globally well-posed. For instance, when the domain is thin in the vertical direction, G. Raugel and G. Sell \cite{RS} were able to establish global solutions for a family of large initial data by using anisotropic Sobolev imbedding theorems (see also the paper \cite{IRS} by D. Iftimie, G. Raugel and
G. Sell). By choosing the initial data to transform
the equation into a rotating fluid equations, A.
Mahalov and B. Nicolaenko \cite{MN} obtained global solutions generated by a family of  large initial data.  A family of axi-symmetric large solutions were established in
\cite{TLL} by Thomas Hou, Z. Lei and C. M. Li.  Recently, Z. Lei, F. Lin and Y. Zhou in \cite{LLZ} proved the global well-posedness of 3D Navier-Stokes equations for a family of large initial data by making use of the structure of Helicity. The data in \cite{LLZ} are not small in $\dot{B}^{-1}_{\infty,\infty}$ even in the anisotropic sense. We also mention that for the general 3D incompressible Navier-Stokes equations which possess hyper-dissipation in horizontal direction, D. Fang and B. Han in \cite{DH} obtain the global existence result when the initial data
 belongs to the anisotropic Besov spaces.

The remaining part of the paper is organized as follows. In Section 2, we present the
basic theories of anisotropic Littlewood-Paley decomposition and anisotropic Besov spaces. Section 3 is devoted to obtaining the $a\ priori$ estimates of solution. The $\theta(t)$ will be studied
in Section 4. Finally, the proof of the main result will be given in Section 5.

\section{Anisotropic Littlewood-Paley theories and preliminary lemmas}
In this section, we first recall the definition of  the  anisotropic Littlewood-Paley decomposition and some properties about anisotropic Besov spaces. It was introduced
  by D. Iftimie in \cite{If} for the study of incompressible Navier-Stokes equations in thin domains. Let us briefly
  explain how this may be built in $\Real^n$. Let $(\chi,\varphi)$ be a couple of $C^\infty$ functions satisfying
$$\hbox{Supp}\chi\subset\{r\leq\frac{4}{3}\},
\ \ \ \
\hbox{Supp}\varphi\subset\{\frac{3}{4}\leq r\leq\frac{8}{3}\},
$$
and
$$\chi(r)+\sum_{k\in \mathbb{N}}\varphi(2^{-k}r)=1\ \ \mathrm{for}\ \ r\in\Real,$$
$$\sum_{j\in \mathbb{Z}}\varphi(2^{-j}r)=1\ \ \mathrm{for}\ \ r\in\Real\backslash \{0\}.$$
For $u\in\mathcal {S}'(\Real^n)/\mathcal{P}(\Real^n)$, we define the homogeneous dyadic decomposition on the horizontal variables by
$${{\Delta}}^h_ku=\mathcal {F}^{-1}(\varphi(2^{-k}|\xi_h|)\widehat{u})\ \ \hbox{for}\ \ k\in\mathbb{Z}.$$
Similarly, on the vertical variable, we define the homogeneous dyadic decomposition by
$${{{\Delta}}}^v_{j}u=\mathcal {F}^{-1}(\varphi(2^{-j}|\xi_n|)\widehat{u})\ \ \hbox{for}\ \ j\in\mathbb{Z}.$$
The anisotropic Littlewood-Paley decomposition satisfies the property of almost orthogonality:
$${{{\Delta}}}_k^h{{{\Delta}}}_l^h u\equiv0\quad \mathrm{if}\quad |k-l|\geq2\quad\quad\mathrm{and}\quad\quad{{{\Delta}}}_k^h(S_{l-1}^hu{{{\Delta}}}_l^hu)\equiv0\quad\mathrm{if}\quad|k-l|\geq5,$$
where $S_{l}^h$ is defined by
$$S_{l}^hu=\sum\limits_{l'\leq l-1}{{\Delta}}_{l'}^hu.$$
Similar properties hold for ${{{\Delta}}}_{j}^v$.
In this paper, we shall use the following anisotropic version of Besov spaces \cite{If}. In what follows, we denote for abbreviation $$\Delta_{k,j}f\overset{\text{def}}={{\Delta}}_k^h{{\Delta}}_j^vf.$$
\begin{defn}[Anisotropic Besov space]\label{D2.1}
Let $(p,r)\in[1,\infty]^2$, $\sigma,s\in\Real$ and $u\in\mathcal {S}'(\Real^n)/\mathcal{P}(\Real^n),$ we set

$$\|u\|_{\dot{B}^{\sigma,s}_{p,r}}\overset{\mathrm{def}}=\|2^{k\sigma}2^{js}\|{\Delta}_{k,j}u\|_{L^p_h(L_v^2)}\|_{l^r(\mathbb{Z}^2)}.$$
(1) For $\sigma<\frac{n-1}{p},s<\frac{1}{2}$( $\sigma=\frac{n-1}{p}$ or $s=\frac{1}{2}$ if $r=1$), we define
$$\dot{B}^{\sigma,s}_{p,r}(\Real^n)\overset{\mathrm{def}}=\{u\in\mathcal {S}'(\Real^n)\mid \|u\|_{\dot{B}^{\sigma,s}_{p,r}}<\infty \}.$$
(2) If $k,l\in\mathbb{N}$ and $\frac{n-1}{p}+k < \sigma<\frac{n-1}{p}+k+1$, $\frac{1}{2}+l < s<\frac{1}{2}+l+1$ ( $\sigma=\frac{n-1}{p}+k+1$ or $s=\frac{1}{2}+l+1$ if $r=1$),
then $\dot{B}^{\sigma,s}_{p,r}(\Real^n)$ is defined as
the subset of  $u\in\mathcal {S}'(\Real^n)$ such that $\partial_h^\beta \partial_3^\alpha u\in \dot{B}^{\sigma-k,s-l}_{p,r}(\Real^n)$
whenever $|\beta|=k, \alpha=l.$
\end{defn}
The study of non-stationary equation requires spaces of the type
$L^\rho_T(X)$ for appropriate Banach spaces $X$. In
our case, we expect $X$ to be an anisotropic Besov space. So it
is natural to localize the equations through anisotropic Littlewood-Paley
decomposition. We then get estimates for each dyadic block and
perform integration in time. As in \cite{Ch2}, we define the so called Chemin-Lerner type spaces:
\begin{defn}\label{D2.3}
Let $(p,r)\in[1,\infty]^2$, $\sigma,s\in\Real$ and $T\in(0,\infty],$ we set
$$\|u\|_{\widetilde{L}^\rho_T(\dot{B}^{\sigma,s}_{p,r})}\overset{\mathrm{def}}=
\|2^{k\sigma}2^{js}\|{\Delta}_{k,j}u\|_{L_T^\rho(L^p_h(L_v^2))}\|_{l^r(\mathbb{Z}^2)}$$
and define the space $\widetilde{L}^\rho_T(\dot{B}^{\sigma,s}_{p,r})(\Real^n)$to be the subset of distributions in
$u\in\mathcal {S}'(0,T)\times\mathcal{S}'(\Real^n)$ with finite $\widetilde{L}^\rho_T(\dot{B}^{\sigma,s}_{p,r})$ norm.
\end{defn}

In order to investigate the continuity properties of the products of two temperate distributions $f$ and $g$ in anisotropic Besov spaces, we
  then recall the isotropic product decomposition which is a simple splitting device going back to the pioneering work
  by J.-M. Bony \cite{Bo}. Let $f,g\in\mathcal {S}'(\Real^n)$,
$$fg=T(f,g)+\widetilde{T}(f,g)+R(f,g),$$
where the paraproducts $T(f,g)$ and $\widetilde{T}(f,g)$ are defined by
$$T(f,g)=\sum\limits_{k\in\mathbb{Z}}S_{k-1}f{{\Delta}}_kg,\quad  \widetilde{T}(f,g)=\sum\limits_{k\in\mathbb{Z}}{{\Delta}}_kf S_{k-1}g$$
and the remainder
$$R(f,g)=\sum\limits_{k\in\mathbb{Z}}{{\Delta}}_{k}f\widetilde{{{\Delta}}}_kg\quad\mathrm{with}\quad\widetilde{{{\Delta}}}_kg
=\sum\limits_{k'=k-1}^{k+1}{{\Delta}}_{k'}g.$$
Similarly, we can define the decompositions for both horizontal variable $x_h$ and vertical variable $x_n$.
Indeed, we have the following split in $x_h.$
$$fg=T^h(f,g)+\widetilde{T}^h(f,g)+R^h(f,g),$$
with
$$T^h(f,g)=\sum\limits_{k\in\mathbb{Z}}S^h_{k-1}f{{\Delta}}^h_kg,\quad  \widetilde{T}^h(f,g)=\sum\limits_{k\in\mathbb{Z}}{{\Delta}}^h_kf S^h_{k-1}g$$
and
$$R^h(f,g)=\sum\limits_{k\in\mathbb{Z}}{{\Delta}}^h_{k}f\widetilde{{{\Delta}}}^h_kg\quad\mathrm{where}
\quad\widetilde{{{\Delta}}}^h_kg=\sum\limits_{k'=k-1}^{k+1}{{\Delta}}^h_{k'}g.$$
The decomposition in vertical variable $x_n$ can be defined by the same line. Thus, we can write $fg$ as
\begin{align}\label{e2.3}
\begin{split}
fg&=(T^h+\widetilde{T}^h+R^h)(T^v+\widetilde{T}^v+R^v)(f,g)\\
&=T^hT^v(f,g)+T^h\widetilde{T}^v(f,g)+T^hR^v(f,g)\\
&\quad+\widetilde{T}^hT^v(f,g)+\widetilde{T}^h\widetilde{T}^v(f,g)+\widetilde{T}^hR^v(f,g)\\
&\quad+R^hT^v(f,g)+R^h\widetilde{T}^v(f,g)+R^hR^v(f,g).
\end{split}
\end{align}
Each term of (\ref{e2.3}) has an explicit definition. Here
 $$T^hT^v(f,g)=\sum\limits_{(k,j)\in \mathbb{Z}^2}S_{k-1}^hS_{j-1}^vf{\Delta}_{k,j}g,\quad \widetilde{T}^hT^v(f,g)=\sum\limits_{(k,j)\in \mathbb{Z}^2}{{\Delta}}_{k}^hS_{j-1}^vfS_{k-1}^h{{\Delta}}_j^vg.$$
 Similarly,
 $$R^hT^v(f,g)=\sum\limits_{(k,j)\in \mathbb{Z}^2}{{\Delta}}_{k}^hS_{j-1}^vf\widetilde{{{\Delta}}}_k^h{{\Delta}}_j^vg,\quad R^hR^v(f,g)=\sum\limits_{(k,j)\in \mathbb{Z}^2}\Delta_{k,j} f\widetilde{\Delta}_{k,j}g,$$
 and so on.

At this moment, we can state an important product law in anisotropic Besov spaces. The case $p=2, n=3$ was proved in \cite{GHZ}.  For completeness, here we prove a similar result in the $L^p$ framework.
\begin{lem}\label{p2.5}
Let $1\leq p<n-1$ and  $(\sigma_1,\sigma_2)$ be in $\Real^2$.
If $\sigma_1,\sigma_2\leq \frac{n-1}{p}$ and
$$\aligned
&\sigma_1+\sigma_2>(n-1)\max(0,\frac{2}{p}-1),
\endaligned$$
then we have for any $f\in \dot{B}^{\sigma_1,\frac{1}{2}}_{p,1}(\mathbb{R}^n)$ and $g\in \dot{B}^{\sigma_2,\frac{1}{2}}_{p,1}(\mathbb{R}^n)$,
\begin{align}\label{e2.4}
\|fg\|_{\dot{B}^{\sigma_1+\sigma_2-\frac{n-1}{p},\frac{1}{2}}_{p,1}}\lesssim \|f\|_{\dot{B}^{\sigma_1,\frac{1}{2}}_{p,1}}\|g\|_{\dot{B}^{\sigma_2,\frac{1}{2}}_{p,1}}.
\end{align}
\end{lem}
\begin{proof}
According to (\ref{e2.3}), we first give the bound of $T^hT^v(f,g)$. Indeed, applying H\"{o}lder and Bernstein inequality, we get that
$$\aligned
\|{\Delta}_{k,j}(T^hT^v(f,g))&\|_{L_h^p(L^2_v)}\lesssim  \sum\limits_{\substack{|k-k'|\leq4\\|j-j'|\leq4}}
\|S^h_{k'-1}S^v_{j'-1}f\|_{L^\infty}\|\Delta_{k',j'}g\|_{L_h^p(L^2_v)}\\
&\lesssim  \sum\limits_{\substack{|k-k'|\leq4\\|j-j'|\leq4}} \sum\limits_{\substack{k''\leq k'-2\\j''\leq j'-2}}2^{k''\sigma_1}2^{\frac{1}{2}j''}\|\Delta_{k'',j''}f\|_{L_h^p(L^2_v)}
\cdot2^{k'\sigma_2}2^{\frac{1}{2}j'}\|\Delta_{k',j'}g\|_{L_h^p(L^2_v)}\\
&\quad\times 2^{(k''-k')(\frac{n-1}{p}-\sigma_1)}2^{(k-k')(\sigma_1+\sigma_2-\frac{n-1}{p})}2^{\frac{1}{2}(j-j')}2^{-k(\sigma_1+\sigma_2-\frac{n-1}{p})}2^{-\frac{1}{2}j}.
\endaligned$$
Since $\sigma_1 \leq \frac{n-1}{p}$, we obtain that
$$\aligned
\|{\Delta}_{k,j}(T^hT^v(f,g))\|_{L_h^p(L^2_v)}&\lesssim  c_{k,j}2^{-k(\sigma_1+\sigma_2-\frac{n-1}{p})}2^{-\frac{1}{2}j}\|f\|_{\dot{B}^{\sigma_1,\frac{1}{2}}_{p,1}}\|g\|_{\dot{B}^{\sigma_2,\frac{1}{2}}_{p,1}},
\endaligned$$
where the sequence $\{c_{k,j}\}_{(k,j)\in\mathbb{Z}^2}$ satisfies $\|c_{k,j}\|_{l^1(\mathbb{Z}^2)}=1$.
This gives the estimate of $T^hT^v(f,g)$.

Similarly, for $\widetilde{T}^hT^v(f,g)$,  we have
$$\aligned
\|{\Delta}_{k,j}(\widetilde{T}^hT^v(f,g))&\|_{L_h^p(L^2_v)}\lesssim  \sum\limits_{\substack{|k-k'|\leq4\\|j-j'|\leq4}}
\|{{\Delta}}^h_{k'}S^v_{j'-1}f\|_{L^p_h(L_v^\infty)}\|S^h_{k'-1}{{\Delta}}^v_{j'}g\|_{L^\infty_h(L_v^2)}\\
&\lesssim  \sum\limits_{\substack{|k-k'|\leq4\\|j-j'|\leq4}} \sum\limits_{\substack{k''\leq k'-2\\j''\leq j'-2}}2^{k'\sigma_1}2^{\frac{1}{2}j''}\|\Delta_{k',j''}f\|_{L^p_h(L^2_v)}
\cdot2^{k''\sigma_2}2^{\frac{1}{2}j'}\|\Delta_{k'',j'}g\|_{L^p_h(L^2_v)}\\
&\quad\times 2^{(k''-k')(\frac{n-1}{p}-\sigma_2)}2^{(k-k')(\sigma_1+\sigma_2-\frac{n-1}{p})}2^{\frac{1}{2}(j-j')}2^{-k(\sigma_1+\sigma_2-\frac{n-1}{p})}2^{-\frac{1}{2}j}.
\endaligned$$
Again, $\sigma_2\leq \frac{n-1}{p}$ implies that
$$\aligned
\|{\Delta}_{k,j}(\widetilde{T}^hT^v(f,g))\|_{L_h^p(L^2_v)}&\lesssim  c_{k,j}2^{-k(\sigma_1+\sigma_2-\frac{n-1}{p})}2^{-\frac{1}{2}j}\|f\|_{\dot{B}^{\sigma_1,\frac{1}{2}}_{p,1}}\|g\|_{\dot{B}^{\sigma_2,\frac{1}{2}}_{p,1}}.
\endaligned$$

The estimate on the remainder operator which concerns  the horizontal variable $R^hT^v(f,g)$ may be more complicated.
When $2\leq p$, the strategy is following:
$$\aligned
\|{\Delta}_{k,j}(R^hT^v(f,g))&\|_{L_h^p(L^2_v)}\lesssim  \sum\limits_{\substack{k'\geq k-2\\|j-j'|\leq4}}
\|{{\Delta}}^h_{k'}S^v_{j'-1}f\widetilde{{{\Delta}}}^h_{k'}{{\Delta}}^v_{j'}g\|_{L^{\frac{p}{2}}_h(L_v^2)}2^{\frac{n-1}{p}k}\\
&\lesssim  \sum\limits_{\substack{k'\geq k-2\\|j-j'|\leq4}} \sum\limits_{j''\leq j'-2}2^{k'\sigma_1}2^{\frac{1}{2}j''}\|\Delta_{k',j''}f\|_{L^{p}_h(L_v^2)}
\cdot2^{k'\sigma_2}2^{\frac{1}{2}j'}\|\widetilde{{{\Delta}}}^h_{k'}{{\Delta}}^v_{j'}g\|_{L^p_h(L_v^2)}\\
&\quad\times2^{(k-k')(\sigma_1+\sigma_2)}2^{\frac{1}{2}(j-j')}2^{-k(\sigma_1+\sigma_2-\frac{n-1}{p})}2^{-\frac{1}{2}j}.
\endaligned$$
As $\sigma_1+\sigma_2>0$ if $2\leq p$, we have
$$\aligned
\|{\Delta}_{k,j}(R^hT^v(f,g))\|_{L_h^p(L^2_v)}&\lesssim  c_{k,j}2^{-k(\sigma_1+\sigma_2-\frac{n-1}{p})}2^{-\frac{1}{2}j}\|f\|_{\dot{B}^{\sigma_1,\frac{1}{2}}_{p,1}}\|g\|_{\dot{B}^{\sigma_2,\frac{1}{2}}_{p,1}}.
\endaligned$$
In the case $1\leq p <2$, we have
$$\aligned
\|{\Delta}_{k,j}(R^hT^v(f,g))&\|_{L_h^p(L^2_v)}\lesssim  \sum\limits_{\substack{k'\geq k-2\\|j-j'|\leq4}}
\|{{\Delta}}^h_{k'}S^v_{j'-1}f\widetilde{{{\Delta}}}^h_{k'}{{\Delta}}^v_{j'}g\|_{L^{1}_h(L_v^2)}2^{(n-1)(1-\frac{1}{p})k}\\
&\lesssim  \sum\limits_{\substack{k'\geq k-2\\|j-j'|\leq4}} \sum\limits_{j''\leq j'-2}2^{\frac{1}{2}j''}\|\Delta_{k',j''}f\|_{L^{2}}
\|\widetilde{{{\Delta}}}^h_{k'}{{\Delta}}^v_{j'}g\|_{L^2}2^{(n-1)(1-\frac{1}{p})k}\\
&\lesssim  \sum\limits_{\substack{k'\geq k-2\\|j-j'|\leq4}} \sum\limits_{j''\leq j'-2}2^{k'\sigma_1}2^{\frac{1}{2}j''}\|\Delta_{k',j''}f\|_{L_h^p(L^2_v)}
\cdot2^{k'\sigma_2}2^{\frac{1}{2}j'}\|\widetilde{{{\Delta}}}^h_{k'}{{\Delta}}^v_{j'}g\|_{L_h^p(L^2_v)}\\
&\quad\times2^{(k-k')(\sigma_1+\sigma_2-(n-1)(\frac{2}{p}-1))}2^{\frac{1}{2}(j-j')}2^{-k(\sigma_1+\sigma_2-\frac{n-1}{p})}2^{-\frac{1}{2}j}.
\endaligned$$
As $\sigma_1+\sigma_2>(n-1)(\frac{2}{p}-1)$ if $1\leq p<2$, we have
$$\aligned
\|{\Delta}_{k,j}(R^hT^v(f,g))\|_{L_h^p(L^2_v)}&\lesssim  c_{k,j}2^{-k(\sigma_1+\sigma_2-\frac{n-1}{p})}2^{-\frac{1}{2}j}\|f\|_{\dot{B}^{\sigma_1,\frac{1}{2}}_{p,1}}\|g\|_{\dot{B}^{\sigma_2,\frac{1}{2}}_{p,1}}.
\endaligned$$
The other terms can be followed exactly in the same way, here we omit the details. These complete the proof of this lemma.
\end{proof}

Throughout this paper, $\Phi$ denotes a locally bounded function on $\Real^+\times\Real$ which satisfies the following subadditivity (see \eqref{e3.5} for the explicit expression of $\Phi$)
$$\Phi(t,\xi_n)\leq  \Phi(t,\xi_n-\eta_n)+\Phi(t,\eta_n).$$
For any function $f$ in $\mathcal {S}'(0,T)\times\mathcal{S}'(\Real^n)$, we define
\begin{align}\label{ephi}
f_\Phi(t,x_h,x_n)=\mathcal {F}^{-1}\left(e^{\Phi(t,\xi_n)}\hat{f}(t,x_h,\xi_n)\right)
\end{align}
Let us keep the following fact in mind that the map $f\mapsto f^+$ preserves the norm of $L^p_h(L_v^2)$, where $f^+(t,x_h,x_n)$ represents the inverse Fourier transform of $|\hat{f}(t,x_h,\xi_n)|$ on vertical variable, defined as
$$f^+(t,x_h,x_n)\overset{\mathrm{def}}=\mathcal {F}^{-1}|\hat{f}(t,x_h,\xi_n)|.$$
On the basis of these facts, we have the following  weighted inequality as in  Lemma \ref{p2.5}.
\begin{lem}\label{p2.8}
Let $1\leq p<n-1$ and $(\sigma_1,\sigma_2)$ be in $\Real^2$.
If $\sigma_1,\sigma_2\leq \frac{n-1}{p}$ and
$$\aligned
&\sigma_1+\sigma_2>(n-1)\max(0,\frac{2}{p}-1),
\endaligned$$
 then we have for any $f_\Phi \in \dot{B}^{\sigma_1,\frac{1}{2}}_{p,1}(\mathbb{R}^n)$ and $g_\Phi \in \dot{B}^{\sigma_2,\frac{1}{2}}_{p,1}(\mathbb{R}^n)$,
\begin{align}\nonumber
\|(fg)_\Phi\|_{\dot{B}^{\sigma_1+\sigma_2-\frac{n-1}{p},\frac{1}{2}}_{p,1}}\lesssim  \|f_\Phi\|_{\dot{B}^{\sigma_1,\frac{1}{2}}_{p,1}}\|g_\Phi\|_{\dot{B}^{\sigma_2,\frac{1}{2}}_{p,1}}.
\end{align}
\end{lem}
\begin{proof}
We only prove the  $\dot{B}^{\sigma_1+\sigma_2-\frac{n-1}{p},\frac{1}{2}}_{p,1}$ norm of  $T^hT^v(f,g)_\Phi$. For fixed $k,j$, we have
\begin{align}
\begin{split}\nonumber
&\|{\Delta}_{k,j}(T^hT^v(f,g)_\Phi)\|_{L_h^p(L^2_v)}\\
&\lesssim \sum\limits_{\substack{|k-k'|\leq 4\\|j-j'|\leq 4}}
\|e^{\Psi(t,\xi_n)}\mathcal {F}(S^h_{k'-1} S^v_{j'-1}f)(x_h,\cdot)\star\mathcal {F}(\Delta_{k',j'}g)(x_h,\cdot)\|_{L_h^p(L^2_v)}\\
&\lesssim  \sum\limits_{\substack{|k-k'|\leq 4\\|j-j'|\leq 4}}
\||\mathcal {F}(S^h_{k'-1} S^v_{j'-1}f_\Phi)(x_h,\cdot)|\star|\mathcal {F}(\Delta_{k',j'}g_\Phi)(x_h,\cdot)|\|_{L_h^p(L^2_v)}\\
&\lesssim  \sum\limits_{\substack{|k-k'|\leq 4\\|j-j'|\leq 4}}\sum\limits_{\substack{k''\leq k'-2\\j''\leq j '-2}}
\||\mathcal {F}(\Delta_{k'',j''} f_\Phi)(x_h,\cdot)|\star|\mathcal {F}(\Delta_{k',j'}g_\Phi)(x_h,\cdot)|\|_{L_h^p(L^2_v)}\\\
&\lesssim  \sum\limits_{\substack{|k-k'|\leq 4\\|j-j'|\leq 4}}\sum\limits_{\substack{k''\leq k'-2\\j''\leq j '-2}}
\|(\Delta_{k'',j''}f_\Phi)^+\|_{L_h^p(L^2_v)}\|(\Delta_{k',j'}g_\Phi)^+\|_{L_h^p(L^2_v)}2^{\frac{n-1}{p}k''}2^{\frac{1}{2}j''}.
\end{split}
\end{align}
Using the fact that $f\mapsto f^+$ preserves the norm of  $L_h^p(L^2_v)$, we then get by the similar method as in Lemma \ref{p2.5} that
\begin{align}
\begin{split}\nonumber
\|(T^hT^v(f,g)_\Phi)\|_{\dot{B}^{\sigma_1+\sigma_2-\frac{n-1}{p},\frac{1}{2}}_{p,1}}
\lesssim \|f_\Phi\|_{\dot{B}^{\sigma_1,\frac{1}{2}}_{p,1}}
\|g_\Phi\|_{\dot{B}^{\sigma_2,\frac{1}{2}}_{p,1}}.
\end{split}
\end{align}
The other terms in (\ref{e2.3}) can be estimated by the same method and finally, we have
\begin{align}
\begin{split}\nonumber
\|(fg)_\Phi\|_{\dot{B}^{\sigma_1+\sigma_2-\frac{n-1}{p},\frac{1}{2}}_{p,1}}
\lesssim \|f_\Phi\|_{\dot{B}^{\sigma_1,\frac{1}{2}}_{p,1}}
\|g_\Phi\|_{\dot{B}^{\sigma_2,\frac{1}{2}}_{p,1}}.
\end{split}
\end{align}
\end{proof}
The following lemma  is a direct consequence of Lemma  \ref{p2.8}.
\begin{lem}\label{p2.9}
Let $1\leq p<n-1$, $\rho\in[1,\infty]$, $(\rho_1,\rho_2)\in[1,\infty]^2$ and $(\sigma_1,\sigma_2)$ be in $\Real^2$. Assume that
$$\frac{1}{\rho}\overset{\mathrm{def}}=\frac{1}{\rho_1}+\frac{1}{\rho_2}.
$$
If $\sigma_1,\sigma_2\leq \frac{n-1}{p}$ and
$$\aligned
&\sigma_1+\sigma_2>(n-1)\max(0,\frac{2}{p}-1),
\endaligned$$
 then we have for any $f_\Phi \in \widetilde{L}_T^{\rho_1}(\dot{B}^{\sigma_1,\frac{1}{2}}_{p,1}(\mathbb{R}^n))$ and $g_\Phi \in \widetilde{L}_T^{\rho_2}(\dot{B}^{\sigma_2,\frac{1}{2}}_{p,1}(\mathbb{R}^n))$,
\begin{align}\nonumber
\|(fg)_\Phi\|_{\widetilde{L}_T^{\rho}(\dot{B}^{\sigma_1+\sigma_2-\frac{n-1}{p},\frac{1}{2}}_{p,1})}\lesssim  \|f_\Phi\|_{\widetilde{L}_T^{\rho_1}(\dot{B}^{\sigma_1,\frac{1}{2}}_{p,1})}\|g_\Phi\|_{\widetilde{L}_T^{\rho_2}(\dot{B}^{\sigma_2,\frac{1}{2}}_{p,1})}.
\end{align}
\end{lem}


\section{ Estimates for the re-scaled system}
This section is devoted to obtaining the $a\ priori$ estimate for the following  system
\begin{equation}\label{e3.1}
\left\{
 \begin{array}{rlll}
 \partial_tv^h+ v\cdot\nabla {v}^h+D_\epsilon^s v^h+\nabla_h{q}&=&0,\\
 \partial_tv^n+{v}\cdot\nabla v^n+D_\epsilon^sv^n+\eps^2 \partial_nq&=&0,\\
  \hbox{div}\,{v}&=&0,\\
  {v}(0)&=&{v}_0(x).
   \end{array}
  \right.
\end{equation}
The pressure ${q}$ can be computed by the formula
$$-\Delta_\eps{q}=\sum\limits_{i,j}\partial_i\partial_j({v}^i{v}^j).$$
Due to the divergence free condition, the pressure can be split into the following three parts
\begin{equation}\label{e3.2}
\left\{
 \begin{array}{llll}
q^1=(-\Delta_\eps)^{-1}\sum\limits_{i,j=1}^{n-1}\partial_i\partial_j(v^iv^j),\\
q^2=2(-\Delta_\eps)^{-1}\sum\limits_{i=1}^{n-1}\partial_i\partial_n(v^iv^n),\\
q^3=-2(-\Delta_\eps)^{-1}\partial_n(v^n\mathrm{div_h}{v}^h).
 \end{array}
  \right.
\end{equation}
It is worthwhile to note that there will lose one vertical derivative owing to the term $v^n\partial_n {v}^h$ and pressure terms ${q}^2,{q}^3$ which appear in the equation on ${v}^h$. Thus, we assume that the initial data is analytic in the vertical variable. This method was introduced in \cite{Ch} to compensate the losing derivative in $x_n$.
Therefore, we introduce two key quantities which we want to control in order to obtain the global bound of ${v}$ in a certain space. We
define the function $\theta(t)$ by
\begin{equation}\label{e3.3}
\begin{split}
\theta(t)=\int_0^t \|{v}^n_{\Phi}(\tau)\|_{\dot{B}^{\frac{n-1}{p},\frac{1}{2}}_{p,1}}d\tau,
\end{split}
\end{equation}
and denote
\begin{equation}\label{e3.4}
\begin{split}
\Psi(t)=&\|{v}_\Phi\|_{\widetilde{L}_t^\infty(\dot{B}_{p,1}^{\frac{n-1}{p}+1-s,\frac{1}{2}})}
+\|{v}_\Phi\|_{L_t^1(\dot{B}_{p,1}^{\frac{n-1}{p}+1,\frac{1}{2}})}\\
&+\int_0^t \|v^n_\Phi\|_{\dot{B}_{p,1}^{\frac{n-1}{p}+1,\frac{1}{2}}}\|\partial_nv^h_\Phi\|_{\dot{B}_{p,1}^{\frac{n-1}{p}+1-s,\frac{1}{2}}} d\tau,\\
\Psi(0)=&\|e^{\alpha D_n }v_0\|_{\dot{B}_{p,1}^{\frac{n-1}{p}+1-s,\frac{1}{2}}}.
\end{split}
\end{equation}
The phase function $\Phi(t,D_n)$ is defined by
\begin{equation}\label{e3.5}
\begin{split}
\Phi(t,\xi_n)=(\alpha-\lambda\theta(t))|\xi_n|,
\end{split}
\end{equation}
for some $\lambda $ that will be chosen later on, $\alpha$ is a positive number. Obviously, we need to ensure that $\theta(t)< \frac{\alpha}{\lambda}$ which
implies the subadditivity of $\Phi$.

The following lemma provides the $a\ priori$ estimate of ${v}_\Phi$ in the anisotropic Besov spaces, which is the key bilinear estimate.
\begin{lem}\label{l3.2}
There exist two constants $\lambda_0$ and $C_1$ such that for any $\lambda>\lambda_0$ and $t$ satisfying $\theta(t)\leq \frac{\alpha}{2\lambda}$, we have
$$\Psi(t)\leq C_1\Psi(0) +  C_1\Psi(t)^2.$$
\end{lem}

\subsection{Estimates on the horizontal component ${v}^h$}
According to the definition of ${v}_\Phi^h$, we find that in each dyadic block,  it verifies the following equation
\begin{align}\label{e3.6}
\begin{split}
{\Delta}_{k,j}v_\Phi^h(t,x)&=e^{-tD_\eps^s+\Phi(t,D_n)}{\Delta}_{k,j}v^h_0\\
&-\int_0^te^{-(t-\tau)D_\eps^s}e^{-\lambda D_n
\int_\tau^t\dot\theta(t')dt'}{\Delta}_{k,j}(v\cdot\nabla v^h)_\Phi(\tau)d\tau\\
&-\int_0^te^{-(t-\tau)D_\eps^s}e^{-\lambda D_n \int_\tau^t\dot\theta(t')dt'}\nabla_h{\Delta}_{k,j}q_\Phi(\tau)d\tau.
\end{split}
\end{align}
Taking the $L_h^p(L_v^2)$ norm, we deduce that
 \begin{align}\label{e3.8}
\begin{split}
\|{\Delta}_{k,j}v_\Phi^h\|_{L_h^p(L_v^2)}&\lesssim  e^{-c(2^{ks}+\eps^s2^{js})t}\|{\Delta}_{k,j}e^{\alpha D_n }v_0^h\|_{L_h^p(L_v^2)}\\
&+\int_0^te^{-c(2^{ks}+\eps^s2^{js})(t-\tau)}e^{-c\lambda2^{j}\int_\tau^t\dot\theta(t')dt'}\|{\Delta}_{k,j}(v\cdot\nabla v^h)_\Phi\|_{L_h^p(L_v^2)}d\tau\\
&+\int_0^te^{-c(2^{ks}+\eps^s2^{js})(t-\tau)}e^{-c\lambda2^{j}\int_\tau^t\dot\theta(t')dt'}\|\nabla_h{\Delta}_{k,j}q_\Phi\|_{L_h^p(L_v^2)}d\tau\\
&\overset{\mathrm{def}}{=}I_1+I_2+I_3.
\end{split}
\end{align}

We first estimate the linear term $I_1$. In fact, we have
 \begin{align}\label{e3.9}
\begin{split}
\| I_1\|_{L^\infty_t}+2^{ks}\| I_1\|_{L^1_t}&\lesssim \|{\Delta}_{k,j}e^{\alpha  D_n }v_0^h\|_{L_h^p(L_v^2)}\\
&\lesssim  c_{k,j}2^{-k(\frac{n-1}{p}+1-s)}2^{-\frac{1}{2}j}\|e^{\alpha D_n }v_0^h\|_{\dot{B}_{p,1}^{\frac{n-1}{p}+1-s,\frac{1}{2}}},
\end{split}
\end{align}
where $\{c_{k,j}\}_{(k,j)\in\mathbb{Z}^2}$ is a two dimensional sequence satisfying $\|c_{k,j}\|_{l^1(\mathbb{Z}^2)}=1.$

The term $I_{2}$ can be rewritten as
$$\aligned
I_{2}&\lesssim \int_0^te^{-c2^{ks}(t-\tau)}\|{\Delta}_{k,j}(v\cdot\nabla v^h)_\Phi\|_{L^p_h(L^2_v)}d\tau.
\endaligned$$
By Young's inequality, we have
$$\aligned
\| I_{2}\|_{L^\infty_t}+ 2^{ks}\|I_{2}\|_{L^1_t} &\lesssim \|\Delta_{k,j}(v\cdot\nabla v^h)_\Phi\|_{L^1_t(L_h^p(L_v^2))}\\
&\lesssim  c_{k,j}2^{-k(\frac{n-1}{p}+1-s)}2^{-\frac{1}{2}j}\|(v\cdot\nabla v^h)_\Phi\|_{L_t^1(\dot{B}_{p,1}^{\frac{n-1}{p}+1-s,\frac{1}{2}})}.
\endaligned$$
Thus, we can get by Lemma \ref{p2.8} and \ref{p2.9}  that
 \begin{align}\nonumber
\begin{split}
\| I_{2}\|_{L^\infty_t}+ 2^{ks}\|I_{2}\|_{L^1_t}&\lesssim c_{k,j}2^{-k(\frac{n-1}{p}+1-s)}2^{-\frac{1}{2}j}\|v^h_\Phi\|_{\widetilde{L}_t^\infty(\dot{B}_{p,1}^{\frac{n-1}{p}+1-s,\frac{1}{2}})}
\|v^h_\Phi\|_{L_t^1(\dot{B}_{p,1}^{\frac{n-1}{p}+1,\frac{1}{2}})}\\
&+ c_{k,j}2^{-k(\frac{n-1}{p}+1-s)}2^{-\frac{1}{2}j}\int_0^t\|v^n_\Phi\|_{ \dot{B}_{p,1}^{\frac{n-1}{p},\frac{1}{2}}} \|\partial_nv^h_\Phi\|_{ \dot{B}_{p,1}^{\frac{n-1}{p}+1-s,\frac{1}{2}}}d\tau.
\end{split}
\end{align}

Now we are left with the study of the pressure term $I_3$.
The pressure can be split into  ${q}={q}^1+{q}^2+{q}^3$ with ${q}^1,{q}^2,{q}^3$ defined in (\ref{e3.2}).
For convenience, we denote that
$$I_{31}=\int_0^te^{-c(2^{ks}+\eps^s2^{js})(t-\tau)}e^{-c\lambda2^{j}\int_\tau^t\dot\theta(t')dt'}\|\nabla_h{\Delta}_{k,j}q^1_\Phi\|_{L^p_h(L^2_v)}d\tau,$$ $$I_{32}=\int_0^te^{-c(2^{ks}+\eps^s2^{js})(t-\tau)}e^{-c\lambda2^{j}\int_\tau^t\dot\theta(t')dt'}\|\nabla_h{\Delta}_{k,j}q^2_\Phi\|_{L^p_h(L^2_v)}d\tau,$$
$$I_{33}=\int_0^te^{-c(2^{ks}+\eps^s2^{js})(t-\tau)}e^{-c\lambda2^{j}\int_\tau^t\dot\theta(t')dt'}\|\nabla_h{\Delta}_{k,j}q^3_\Phi\|_{L^p_h(L^2_v)}d\tau.$$
Hence, using the fact that
$(-\Delta_\eps)^{-1}\partial_i\partial_j$ is a  bounded operator applied for frequency localized
functions in $L^p_h(L^2_v)$ when $i,j=1,2,\cdots,n-1$, we get
$$\|\nabla_h {\Delta}_{k,j}q^1_\Phi\|_{L^p_h(L^2_v)}\lesssim \|{\Delta}_{k,j}(v^h\cdot \nabla_hv^h)\|_{L^p_h(L^2_v)}.$$
By the same method as in the estimate of  $I_{2}$, we have
\begin{align}\label{e3.14}
\begin{split}
&\|I_{31}\|_{L^\infty_t}+ 2^{ks}\|I_{31}\|_{L^1_t}\lesssim c_{k,j}2^{-k(\frac{n-1}{p}+1-s)}2^{-\frac{1}{2}j}\|v^h_\Phi\|_{\widetilde{L}_t^\infty(\dot{B}_{p,1}^{\frac{n-1}{p}+1-s,\frac{1}{2}})}
\|v^h_\Phi\|_{L_t^1(\dot{B}_{p,1}^{\frac{n-1}{p}+1,\frac{1}{2}})}.
\end{split}
\end{align}
Noting that
$$\nabla_hq^2=2(-\Delta_\eps)^{-1}\nabla_h\partial_i(v^n\partial_nv^h-v^h\mathrm{div}_hv^h),$$
and as in the estimate of $I_2$, it holds that
\begin{align}
\begin{split}
\|I_{32}\|_{L_t^\infty}+2^{ks}\|I_{32}\|_{L_t^1}&\lesssim c_{k,j}2^{-k(\frac{n-1}{p}+1-s)}2^{-\frac{1}{2}j}\int_0^t\|v^n_\Phi\|_{ \dot{B}_{p,1}^{\frac{n-1}{p},\frac{1}{2}}} \|\partial_nv^h_\Phi\|_{ \dot{B}_{p,1}^{\frac{n-1}{p}+1-s,\frac{1}{2}}}d\tau\\
&+c_{k,j}2^{-k(\frac{n-1}{p}+1-s)}2^{-\frac{1}{2}j}
\|v^h_\Phi\|_{L_t^1(\dot{B}_{p,1}^{\frac{n-1}{p}+1,\frac{1}{2}})}\|v^h_\Phi\|_{\widetilde{L}_t^\infty(\dot{B}_{p,1}^{\frac{n-1}{p}+1-s,\frac{1}{2}})}.
\end{split}
\end{align}

Using
$$\nabla_hq^3=2(-\Delta_\eps)^{-1}\nabla_h(\mathrm{div}_hv^h\mathrm{div}_hv^h-v^n\partial_n\mathrm{div}_hv^h),$$
we write $I_{33}$ as follows
$$\aligned
I_{33}&\lesssim 2^{-k}\int_0^te^{-c2^{ks}(t-\tau)}\Big(\|\Delta_{k,j}(\mathrm{div}_hv^h \mathrm{div}_hv^h)_\Phi\|_{L^p_h(L^2_v)}+\|\Delta_{k,j}(v^n\partial_n\mathrm{div}_hv^h)_\Phi\|_{L^p_h(L^2_v)}\Big)d\tau.
\endaligned$$
By Young's inequality and Lemma \ref{p2.9}, as $1\leq s < \min\{n-1,2\frac{n-1}{p}\}$, we have
\begin{align}
\begin{split}\label{e3.15}
\|I_{33}\|_{L_t^\infty}+ 2^{ks}\|I_{33}\|_{L_t^1}&\lesssim c_{k,j}2^{-k(\frac{n-1}{p}+1-s)}2^{-\frac{1}{2}j}\|v^h_\Phi\|_{L_t^\infty(\dot{B}_{p,1}^{\frac{n-1}{p}+1-s,\frac{1}{2}})}\|v^h_\Phi\|_{L_t^1(\dot{B}_{p,1}^{\frac{n-1}{p}+1,\frac{1}{2}})}\\
&+c_{k,j}2^{-k(\frac{n-1}{p}+1-s)}2^{-\frac{1}{2}j}\int_0^t\|v^n_\Phi\|_{ \dot{B}_{p,1}^{\frac{n-1}{p},\frac{1}{2}}} \|\partial_nv^h_\Phi\|_{ \dot{B}_{p,1}^{\frac{n-1}{p}+1-s,\frac{1}{2}}}d\tau.
\end{split}
\end{align}

Now we are going to estimate the key quantity
$$\int_0^t\|v^n_\Phi\|_{ \dot{B}_{p,1}^{\frac{n-1}{p},\frac{1}{2}}} \|\partial_nv^h_\Phi\|_{ \dot{B}_{p,1}^{\frac{n-1}{p}+1-s,\frac{1}{2}}}d\tau.$$
According to (\ref{e3.6}) , we find that in each dyadic block $\partial_n{v}_\Phi^h$ verifies
\begin{align}\label{1}
\begin{split}
{\Delta}_{k,j}\partial_nv_\Phi^h(\tau,x)&=e^{-\tau D_\epsilon^s+\Phi(\tau,D_n)}{\Delta}_{k,j}\partial_nv^h_0\\
&-\int_0^\tau e^{-(\tau-t')D_\epsilon^s}e^{-\lambda D_n
\int_{t'}^\tau\dot\theta(t'')dt''}{\Delta}_{k,j}\partial_n(v\cdot\nabla v^h)_\Phi(t')dt'\\
&-\int_0^\tau e^{-(\tau-t')D_\epsilon^s}e^{-\lambda D_n \int_{t'}^\tau\dot\theta(t'')dt''}\nabla_h{\Delta}_{k,j}\partial_nq_\Phi(t')dt'.
\end{split}
\end{align}
Taking the $L^p_h(L^2_v)$ norm on both sides of (\ref{1}), we have
 \begin{align}\label{2}
\begin{split}
\|{\Delta}_{k,j}\partial_nv_\Phi^h\|_{L^p_h(L^2_v)}&\lesssim  e^{-c(2^{ks}+\eps^s2^{js})\tau}e^{-c\lambda2^{j}\int_{0}^\tau \dot\theta(t'')d{t''}}\|{\Delta}_{k,j}e^{\alpha  D_n }\partial_nv_0^h\|_{L^p_h(L^2_v)}\\
&+\int_0^\tau e^{-c(2^{ks}+\eps^s2^{js})(\tau-t')}e^{-c\lambda2^{j}\int_{t'}^\tau \dot\theta(t'')d{t''}}\|{\Delta}_{k,j}\partial_n(v\cdot\nabla v^h)_\Phi\|_{L^p_h(L^2_v)}dt'\\
&+\int_0^\tau e^{-c(2^{ks}+\eps^s2^{js})(\tau-t')}e^{-c\lambda2^{j}\int_{t'}^\tau\dot\theta(t'')d{t''}}\|\nabla_h{\Delta}_{k,j}\partial_nq_\Phi\|_{L^p_h(L^2_v)}dt'.
\end{split}
\end{align}
For fixed $k,j$, multiplying the (\ref{2}) by $\dot \theta(\tau)$ and integrating over $(0,t)$, one can have
 \begin{align}\nonumber
\begin{split}
\int_0^t\dot\theta(\tau)\|{\Delta}_{k,j}\partial_nv_\Phi^h\|_{L^p_h(L^2_v)}d\tau
&\lesssim\int_0^te^{-c\lambda2^{j}\int_{0}^\tau \dot\theta(t'')d{t''}}\dot\theta(\tau)\|e^{\alpha D_n }\partial_n{\Delta}_{k,j}v_0^h\|_{L^p_h(L^2_v)}d\tau\\
&+\int_0^t\int_0^\tau e^{-c\lambda2^{j}\int_{t'}^\tau \dot\theta(t'')d{t''}}2^j\dot\theta(\tau)\|{\Delta}_{k,j}(v\cdot\nabla v^h)_\Phi\|_{L^p_h(L^2_v)}dt'd\tau\\
&+\int_0^t\int_0^\tau e^{-c\lambda2^{j}\int_{t'}^\tau\dot\theta(t'')d{t''}}2^j\dot\theta(\tau)\|\nabla_h{\Delta}_{k,j}q_\Phi\|_{L^p_h(L^2_v)}dt'd\tau\\
&\overset{\mathrm{def}}{=}I_4+I_5+I_6.
\end{split}
\end{align}
The term $I_4$ containing initial data can be bounded by
 \begin{align}\label{3}
\begin{split}
I_4&\lesssim \int_0^t e^{-c\lambda2^{j}\int_{0}^\tau \dot\theta(t'')d{t''}}2^j\dot\theta(\tau)d\tau\|{\Delta}_{k,j}e^{\alpha  D_n }v_0^h\|_{L^p_h(L^2_v)}\\
&\lesssim  \frac{1}{\lambda}c_{k,j}2^{-k(\frac{n-1}{p}+1-s)}2^{-\frac{1}{2}j}\|e^{\alpha D_n }v_0^h\|_{\dot{B}_{p,1}^{\frac{n-1}{p}+1-s,\frac{1}{2}}}.
\end{split}
\end{align}

By Fubini's theorem, the term $I_5$ can be rewritten as
$$\aligned
I_{5}&\lesssim \int_0^t \int_{t'}^t e^{-c\lambda2^{j}\int_{t'}^\tau \dot\theta(t'')dt''}2^j\dot\theta(\tau)d\tau \|{\Delta}_{k,j}(v\cdot\nabla v^h)_\Phi\|_{L^p_h(L^2_v)}dt' \\
&\lesssim \frac{1}{\lambda} \int_0^t \|{\Delta}_{k,j}(v\cdot\nabla v^h)_\Phi\|_{L^p_h(L^2_v)}dt'\\
&\lesssim \frac{1}{\lambda}c_{k,j}2^{-k(\frac{n-1}{p}+1-s)}2^{-\frac{1}{2}j}\|(v\cdot\nabla v^h)_\Phi\|_{L^1_t(\dot{B}_{p,1}^{\frac{n-1}{p}+1-s,\frac{1}{2}})}.
\endaligned$$
Thus, we can get by Lemma \ref{p2.8} and \ref{p2.9} that
 \begin{align}\nonumber
\begin{split}
I_{5}&\lesssim \frac{1}{\lambda}c_{k,j}2^{-k(\frac{n-1}{p}+1-s)}2^{-\frac{1}{2}j}\|v^h_\Phi\|_{\widetilde L^\infty_t(\dot{B}_{p,1}^{\frac{n-1}{p}+1-s,\frac{1}{2}})}\|v^h_\Phi\|_{L^1_t(\dot{B}_{p,1}^{\frac{n-1}{p}+1,\frac{1}{2}})}\\
&+\frac{1}{\lambda} c_{k,j} 2^{-k(\frac{n-1}{p}+1-s)}2^{-\frac{1}{2}j} \int_0^t \|v^n_\Phi\|_{\dot{B}_{p,1}^{\frac{n-1}{p},\frac{1}{2}}}\|\partial_nv^h_\Phi\|_{\dot{B}_{p,1}^{\frac{n-1}{p}+1-s,\frac{1}{2}}}dt'.
\end{split}
\end{align}

As for $I_6$,
for convenience, we denote that
$$I_{61}=\int_0^t\int_0^\tau e^{-c\lambda2^{j}\int_{t'}^\tau \dot\theta(t'')dt''}2^j\dot\theta(\tau)\|\nabla_h{\Delta}_{k,j}q^1_\Phi\|_{L^p_h(L^2_v)}(t')dt'd\tau,$$
$$I_{62}=\int_0^t \int_0^\tau e^{-c\lambda2^{j}\int_{t'}^\tau \dot\theta(t'')dt''}2^j\dot \theta(\tau)\|\nabla_h{\Delta}_{k,j}q^2_\Phi\|_{L^p_h(L^2_v)}(t')dt' d\tau,$$
$$I_{63}=\int_0^t \int_0^\tau e^{-c\lambda2^{j}\int_{t'}^\tau\dot\theta(t'')dt''}2^j\dot \theta(\tau)\|\nabla_h{\Delta}_{k,j}q^3_\Phi\|_{L^p_h(L^2_v)}(t')dt'd\tau.$$
By the same method as in the estimate of  $I_{5}$, we have
\begin{equation}\label{5}
\begin{split}
I_{61}+I_{62}&\lesssim \frac{1}{\lambda}c_{k,j}2^{-k(\frac{n-1}{p}+1-s)}2^{-\frac{1}{2}j}\|v^h_\Phi\|_{\widetilde{L}_t^\infty(\dot{B}_{p,1}^{\frac{n-1}{p}+1-s,\frac{1}{2}})}
\|v^h_\Phi\|_{{L}_t^1(\dot{B}_{p,1}^{\frac{n-1}{p}+1,\frac{1}{2}})}\\
&+\frac{1}{\lambda} c_{k,j} 2^{-k(\frac{n-1}{p}+1-s)}2^{-\frac{1}{2}j} \int_0^t \|v^n_\Phi\|_{\dot{B}_{p,1}^{\frac{n-1}{p},\frac{1}{2}}}\|\partial_nv^h_\Phi\|_{\dot{B}_{p,1}^{\frac{n-1}{p}+1-s,\frac{1}{2}}}dt'.
\end{split}
\end{equation}
Finally,
 $I_{63}$ can be estimated as follows
$$\aligned
I_{63}&\lesssim\int_0^t\int_0^\tau e^{-c\lambda2^{j}\int_{t'}^\tau \dot\theta(t'')dt''}2^j \dot \theta(\tau)
\|{\Delta}_{k,j}(-\Delta_\eps)^{-1}\nabla_h\partial_n(v^n  \mathrm{div}_hv^h)_\Phi\|_{L^p_h(L^2_v)}(t')dt'd\tau\\
&\lesssim2^{-k}\int_0^t\int_{t'}^\tau e^{-c\lambda2^{j}\int_{t'}^\tau \dot\theta(t'')dt''}2^j \dot \theta(\tau)
d\tau \|{\Delta}_{k,j}\partial_n(v^n\mathrm{div}_hv^h)_\Phi\|_{L^p_h(L^2_v)}(t')dt'.
\endaligned$$
Thus, we can obtain that
\begin{align}
\begin{split}\label{6}
I_{63}&\lesssim \frac{1}{\lambda}c_{k,j}2^{-k(\frac{n-1}{p}+1-s)}2^{-\frac{1}{2}j}\|\partial_n(v^n   \mathrm{div}_hv^h)_\Phi\|_{L^1_t(\dot{B}_{p,1}^{\frac{n-1}{p}-s,\frac{1}{2}})}\\
&\lesssim \frac{1}{\lambda}c_{k,j}2^{-k(\frac{n-1}{p}+1-s)}2^{-\frac{1}{2}j}\|v^h_\Phi\|_{\widetilde{L}_t^\infty(\dot{B}_{p,1}^{\frac{n-1}{p}+1-s,\frac{1}{2}})}
\|v^h_\Phi\|_{{L}_t^1(\dot{B}_{p,1}^{\frac{n-1}{p}+1,\frac{1}{2}})}\\
& \ +\frac{1}{\lambda} c_{k,j} 2^{-k(\frac{n-1}{p}+1-s)}2^{-\frac{1}{2}j} \int_0^t \|v^n_\Phi\|_{\dot{B}_{p,1}^{\frac{n-1}{p},\frac{1}{2}}}\|\partial_nv^h_\Phi\|_{\dot{B}_{p,1}^{\frac{n-1}{p}+1-s,\frac{1}{2}}}dt'.
\end{split}
\end{align}
Together with the above estimates on $I_1-I_6$, we get that
\begin{align}
\begin{split}\label{e3.17}
&\|v_\Phi^h\|_{\widetilde{L}^\infty_t(\dot{B}_{p,1}^{\frac{n-1}{p}+1-s,\frac{1}{2}})}
+\|v_\Phi^h\|_{{L}^1_t(\dot{B}_{p,1}^{\frac{n-1}{p}+1,\frac{1}{2}})}+\int_0^t \|v^n_\Phi\|_{\dot{B}_{p,1}^{\frac{n-1}{p},\frac{1}{2}}}\|\partial_nv^h_\Phi\|_{\dot{B}_{p,1}^{\frac{n-1}{p}+1-s,\frac{1}{2}}} d\tau\\
&\lesssim \Psi(0)+\frac{1}{\lambda}\Psi(t)+\Psi(t)^2.
\end{split}
\end{align}

\subsection{Estimates on the vertical component $v^n$}
We begin this part by studying the equation of $v^n$, which is stated as follows
 $$\partial_tv^n+D_\eps^s v^n+{v}\cdot\nabla v^n+\eps^2\partial_nq=0.$$
Observing that in  the above equation, one can expect that there is no loss of derivative in vertical direction. More precisely, due to divergence free condition, the nonlinear term $v^n\partial_nv^n$ can be rewritten as $-v^n\mathrm{div}_h{v}^h$. Thus the estimate on $v^n$ is  different from $v^h.$

 Applying the anisotropic dyadic decomposition operator ${\Delta}_{k,j}$ to the equation of $v^n$,
then in each dyadic block, $v^n$ satisfies
\begin{equation}\label{e3.23}
\partial_t{\Delta}_{k,j}v^n+D_\eps^s{\Delta}_{k,j}v^n
=-{\Delta}_{k,j}(v^h\cdot\nabla_h v^n)+{\Delta}_{k,j}(v^n\mathrm{div}_hv^h)-\eps^2{\Delta}_{k,j}\partial_nq.
\end{equation}
Let us define $G\overset{\mathrm{def}}=v^h\cdot\nabla_h v^n-v^n\mathrm{div}_hv^h$. We write the solution of (\ref{e3.23}) as follows
\begin{align}\label{e3.24}
\begin{split}
{\Delta}_{k,j}{v}^n_{\Phi}&=e^{-tD_\epsilon^s+\Phi(t,D_n)}{\Delta}_{k,j}v^n_0+\int_0^te^{-(t-\tau)D_\epsilon^s}e^{-\lambda D_n \int_\tau^t\dot{\theta}(t')dt'}{\Delta}_{k,j}G_{\Phi}d\tau\\
&\quad+\eps^2\int_0^te^{-(t-\tau)D_\epsilon^s}e^{-\lambda D_n \int_\tau^t\dot{\theta}(t')dt'}{\Delta}_{k,j}\partial_nq_{\Phi}d\tau.
\end{split}
\end{align}
Taking the $L^p_h(L^2_v)$ norm, we infer that
\begin{align}\label{e3.25}
\begin{split}
\|{\Delta}_{k,j}{v}^n_{\Phi}\|_{L^p_h(L^2_v)}&\lesssim e^{-c2^{ks}t}\|{\Delta}_{k,j} e^{\alpha  D_n }v^n_0\|_{L^p_h(L^2_v)}\\
&\quad+\int_0^te^{-c2^{ks}(t-\tau)}\|{\Delta}_{k,j}G_{\Phi}\|_{L^p_h(L^2_v)}d\tau\\
&\quad+\eps^2\int_0^te^{-c2^{ks}(t-\tau)}\|{\Delta}_{k,j}\partial_nq_{\Phi}\|_{L^p_h(L^2_v)}d\tau.
\end{split}
\end{align}
By the Young's inequality, we deduce that
\begin{align}\label{e3.26}
\begin{split}
\|{\Delta}_{k,j}{v}^n_{\Phi}\|_{L_t^\infty(L^p_h(L^2_v))}&+2^{ks}
\|{\Delta}_{k,j}{v}^n_{\Phi}\|_{L_t^1(L^p_h(L^2_v))}\\
&\lesssim\|{\Delta}_{k,j} e^{\alpha  D_n }v^n_0\|_{L^p_h(L^2_v)}
+\|{\Delta}_{k,j}G_{\Phi}\|_{L^1_t(L^p_h(L^2_v))}\\
&\quad+\eps^2\|{\Delta}_{k,j}\partial_nq_{\Phi}\|_{L^1_t(L^p_h(L^2_v))}.
\end{split}
\end{align}
Multiplying both sides of (\ref{e3.26}) by $2^{k(\frac{n-1}{p}+1-s)}2^{j\frac{1}{2}}$ and taking the sum over $k,j$, we have
\begin{align}\label{e3.27}
\begin{split}
&\|{v}^n_{\Phi}\|_{\widetilde{L}_t^\infty(\dot{B}^{\frac{n-1}{p}+1-s,\frac{1}{2}}_{p,1})}+
\|{v}^n_{\Phi}\|_{L_t^1(\dot{B}^{\frac{n-1}{p}+1,\frac{1}{2}}_{p,1})}\\
&\lesssim\|e^{\alpha  D_n }v^n_0\|_{\dot{B}^{\frac{n-1}{p}+1-s,\frac{1}{2}}_{p,1}}
+\|G_{\Phi}\|_{L^1_t(\dot{B}^{\frac{n-1}{p}+1-s,\frac{1}{2}}_{p,1})}
+\eps^2\|\partial_nq_{\Phi}\|_{L^1_t(\dot{B}^{\frac{n-1}{p}+1-s,\frac{1}{2}}_{p,1})}.
\end{split}
\end{align}
According to Lemma \ref{p2.9}, we can obtain the estimates of nonlinear term by the following:
$$\|(v^h\cdot\nabla_h v^n)_{\Phi}\|_{L^1_t(\dot{B}^{\frac{n-1}{p}+1-s,\frac{1}{2}}_{p,1})}\lesssim \|{v}^h_{\Phi}\|_{\widetilde{L}^\infty_t(\dot{B}^{\frac{n-1}{p}+1-s,\frac{1}{2}}_{p,1})}
\|{v}^n_{\Phi}\|_{L^1_t(\dot{B}^{\frac{n-1}{p}+1,\frac{1}{2}}_{p,1})},$$
$$\|(v^n\mathrm{div}_h v^h)_{\Phi}\|_{L^1_t(\dot{B}^{\frac{n-1}{p}+1-s,\frac{1}{2}}_{p,1})}\lesssim \|{v}^n_{\Phi}\|_{\widetilde{L}^\infty_t(\dot{B}^{\frac{n-1}{p}+1-s,\frac{1}{2}}_{p,1})}
\|{v}^h_{\Phi}\|_{L^1_t(\dot{B}^{\frac{n-1}{p}+1,\frac{1}{2}}_{p,1})}.$$
This implies that
$$\|G_{\Phi}\|_{L^1_t(\dot{B}^{\frac{n-1}{p}+1-s,\frac {1}{2}}_{p,1})}\lesssim \|{v}_{\Phi}\|_{\widetilde{L}^\infty_t(\dot{B}^{\frac{n-1}{p}+1-s,\frac {1}{2}}_{p,1})}
\|{v}_{\Phi}\|_{L^1_t(\dot{B}^{\frac{n-1}{p}+1,\frac {1}{2}}_{p,1})}.$$

While for the pressure term, we use the decomposition $q=q^1+q^2+q^3$ in (\ref{e3.2}). For $q^1$, since $\eps(-\Delta_\epsilon)^{-1}\partial_i\partial_n$ is a  bounded operator applied for frequency localized
functions in  $L^p_h(L^2_v)$ if $i =1,2,\cdots,n-1$, we have
$$\aligned
\eps^2\|\partial_nq^1_{\Phi}\|_{L^1_t(\dot{B}^{\frac{n-1}{p}+1-s,\frac {1}{2}}_{p,1})}
&=\eps^2\|(-\Delta_\epsilon)^{-1}\partial_i\partial_j\partial_n
(v^iv^j)_{\Phi}\|_{L^1_t(\dot{B}^{\frac{n-1}{p}+1-s,\frac {1}{2}}_{p,1})}\\
&\lesssim \eps\|\nabla_h(v^hv^h)_{\Phi}\|_{L^1_t(\dot{B}^{\frac{n-1}{p}+1-s,\frac {1}{2}}_{p,1})}.
\endaligned$$
Therefore, we get by using Lemma \ref{p2.9} that
$$\eps^2\|\partial_nq^1_{\Phi}\|_{L^1_t(\dot{B}^{\frac{n-1}{p}+1-s,\frac {1}{2}}_{p,1})}\lesssim \eps\|{v}^h_{\Phi}\|_{\widetilde{L}^\infty_t(\dot{B}^{\frac{n-1}{p}+1-s,\frac{1}{2}}_{p,1})}
\|{v}^h_{\Phi}\|_{L^1_t(\dot{B}^{\frac{n-1}{p}+1,\frac {1}{2}}_{p,1})}.$$
Similarly, the fact that $\eps^2(-\Delta_\epsilon)^{-1}\partial^2_n$ is a bounded operator applied for frequency localized
functions in $L^p_h(L^2_v)$ implies
$$
\eps^2\|\partial_nq^2_{\Phi}\|_{L^1_t(\dot{B}^{\frac{n-1}{p}+1-s,\frac {1}{2}}_{p,1})}\lesssim \|\nabla_h(v^nv^h)_{\Phi}\|_{L^1_t(\dot{B}^{\frac{n-1}{p}+1-s,\frac {1}{2}}_{p,1})},
$$
$$
\eps^2\|\partial_nq^3_{\Phi}\|_{L^1_t(\dot{B}^{\frac{n-1}{p}+1-s,\frac {1}{2}}_{p,1})}\lesssim \|(v^n\mathrm{div}_hv^h)_{\Phi}\|_{L^1_t(\dot{B}^{\frac{n-1}{p}+1-s,\frac {1}{2}}_{p,1})}.
$$
Thus, we have
$$\aligned
\eps^2\|\partial_nq^2_{\Phi}\|_{L^1_t(\dot{B}^{\frac{n-1}{p}+1-s,\frac {1}{2}}_{p,1})}
+\eps^2\|\partial_nq^3_{\Phi}\|_{L^1_t(\dot{B}^{\frac{n-1}{p}+1-s,\frac {1}{2}}_{p,1})}\lesssim \|{v}_{\Phi}\|_{\widetilde{L}^\infty_t(\dot{B}^{\frac{n-1}{p}+1-s,\frac {1}{2}}_{p,1})}
\|{v}_{\Phi}\|_{L^1_t(\dot{B}^{\frac{n-1}{p}+1,\frac {1}{2}}_{p,1})}.
\endaligned$$
Then we obtain that
$$\aligned
\eps^2\|\partial_nq_{\Phi}\|_{L^1_t(\dot{B}^{\frac{n-1}{p}+1-s,\frac {1}{2}}_{p,1})}\lesssim \|{v}_{\Phi}\|_{\widetilde{L}^\infty_t(\dot{B}^{\frac{n-1}{p}+1-s,\frac {1}{2}}_{p,1})}
\|{v}_{\Phi}\|_{L^1_t(\dot{B}^{\frac{n-1}{p}+1,\frac {1}{2}}_{p,1})}.
\endaligned$$
 Combining the above estimates,  we can get the bound of ${v}^n_{\Phi}$  as follows:
\begin{align}\label{e3.28}
\begin{split}
\|{v}^n_{\Phi}\|_{\widetilde{L}_t^\infty(\dot{B}^{\frac{n-1}{p}+1-s,\frac {1}{2}}_{p,1})}+
\|{v}^n_{\Phi}\|_{L_t^1(\dot{B}^{\frac{n-1}{p}+1,\frac {1}{2}}_{p,1})}
\lesssim  \Psi(0) + \Psi(t)^2.
\end{split}
\end{align}

Together (\ref{e3.28}) with (\ref{e3.17}), we finally get that
\begin{align}\label{e3.35}
\begin{split}
\Psi(t) \lesssim \Psi(0)+\frac{1}{\lambda}\Psi(t)+\Psi(t)^2.
\end{split}
\end{align}
This completes the proof of Lemma \ref{l3.2} by choosing $\lambda$ large enough.

\section{Estimates for $\theta(t)$ }
 In the above section, we have used the fact that $\Phi(t)$ is a subadditivity function. This means we should ensure that $\theta(t)< \frac{\alpha}{\lambda}$. Thus,
 it is sufficient to prove that for any time $t$, $\theta(t)$ is a small quantity. By the definition of $\theta(t)$, naturally, we assume that $e^{\alpha D_n }v^n_0$ belongs to $\dot{B}_{p,1}^{\frac{n-1}{p}-s,\frac{1}{2}}$. According to the property of the operator $\partial_t + D_\epsilon ^s$, then we can get the bound for ${v}^n_{\Phi}$ in ${L}_t^1(\dot{B}_{p,1}^{\frac{n-1}{p},\frac{1}{2}})$.
  However, we can not enclose the estimate for ${v}^n_{\Phi}$ in $\widetilde{L}_t^\infty(\dot{B}_{p,1}^{\frac{n-1}{p}-s,\frac{1}{2}})\cap{L}_t^1(\dot{B}_{p,1}^{\frac{n-1}{p},\frac{1}{2}})$.
  Our observation is to add
an extra term $\epsilon v^h$ under the same norm which is hidden in the pressure term $\eps^2\partial_nq^1.$
Hence,  we first
denote that $$X_0=\eps\|e^{\alpha D_n }v^h_0\|_{\dot{B}_{p,1}^{\frac{n-1}{p}-s,\frac{1}{2}}},$$
$$Y_0=\|e^{\alpha D_n }v^n_0\|_{\dot{B}_{p,1}^{\frac{n-1}{p}-s,\frac{1}{2}}},$$
$$X(t)=\eps\|{v}_\Phi^h\|_{\widetilde{L}_t^\infty(\dot{B}_{p,1}^{\frac{n-1}{p}-s,\frac{1}{2}})}+\eps\|{v}_\Phi^h\|_{{L}_t^1(\dot{B}_{p,1}^{\frac{n-1}{p},\frac{1}{2}})},$$
$$Y(t)=\|{v}^n_{\Phi}\|_{\widetilde{L}_t^\infty(\dot{B}_{p,1}^{\frac{n-1}{p}-s,\frac{1}{2}})}
+\|{v}^n_{\Phi}\|_{{L}_t^1(\dot{B}_{p,1}^{\frac{n-1}{p},\frac{1}{2}})}.$$
In order to get the  desired estimates, it suffices to prove the following lemma.
\begin{lem}\label{l3.1}
There exists a constant $C_2$ such that for any $\lambda>0$ and $t$ satisfying $\theta(t)\leq \frac{\alpha}{2\lambda}$, we have
$$X(t)+Y(t)\leq C_2(X_0+Y_0)+C_2(X(t)+Y(t))\Psi(t).$$
\end{lem}
\begin{proof}
We apply the same method as in the above section to prove $Y(t).$ Indeed,
multiplying both sides of (\ref{e3.26}) by $2^{(\frac{n-1}{p}-s)k}2^{\frac{1}{2}j}$ and taking the sum over $k,j$, we can get that
\begin{align}\label{e3.37}
\begin{split}
\|{v}^n_{\Phi}\|_{\widetilde{L}_t^\infty(\dot{B}^{\frac{n-1}{p}-s,\frac{1}{2}}_{p,1})}+
\|{v}^n_{\Phi}\|_{{L}_t^1(\dot{B}^{\frac{n-1}{p},\frac{1}{2}}_{p,1})}&\leq\|e^{\alpha D_n }v^n_0\|_{\dot{B}^{\frac{n-1}{p}-s,\frac{1}{2}}_{p,1}}\\
&\quad+\|G_{\Phi}\|_{{L}^1_t(\dot{B}^{\frac{n-1}{p}-s,\frac{1}{2}}_{p,1})}
+\eps^2\|\partial_nq_{\Phi}\|_{{L}^1_t(\dot{B}^{\frac{n-1}{p}-s,\frac{1}{2}}_{p,1})}.
\end{split}
\end{align}
According to Lemma \ref{p2.9}, we can obtain the estimates of nonlinear terms that
$$\aligned\|\mathrm{div}_h(v^n{v}^h)_{\Phi}\|_{{L}^1_t(\dot{B}^{\frac{n-1}{p}-s,\frac{1}{2}}_{p,1})}
&+\|(v^n\mathrm{div}_h{v}^h)_{\Phi}\|_{L^1_t(\dot{B}^{\frac{n-1}{p}-s,\frac{1}{2}}_{p,1})}\\
&\lesssim  \|{v}^n_{\Phi}\|_{{L}^1_t(\dot{B}^{\frac{n-1}{p},\frac{1}{2}}_{p,1})}
\|{v}^h_{\Phi}\|_{\widetilde{L}^\infty_t(\dot{B}^{\frac{n-1}{p}+1-s,\frac{1}{2}}_{p,1})}.
\endaligned$$
This implies that
$$\|G_{\Phi}\|_{L^1_t(\dot{B}^{\frac{n-1}{p}-s,\frac{1}{2}}_{p,1})}\lesssim  Y(t)\Psi(t).$$

While for the pressure term, we use the decomposition $q=q^1+q^2+q^3$ in (\ref{e3.2}). For $q_1$, since $\eps(-\Delta_\eps)^{-1}\partial_i\partial_n$ is a  bounded operator applied for frequency localized
functions in  $L^p_h(L^2_v)$ if $i =1,2,\cdots,n-1$, we have
$$\aligned
\eps^2\|\partial_nq^1_{\Phi}\|_{{L}^1_t(\dot{B}^{\frac{n-1}{p}-s,\frac{1}{2}}_{p,1})}
&=\eps^2\|(-\Delta_\eps)^{-1}\partial_i\partial_j\partial_n
(v^iv^j)_{\Phi}\|_{{L}^1_t(\dot{B}^{\frac{n-1}{p}-s,\frac{1}{2}}_{p,1})}\\
&\lesssim \eps\|\nabla_h(v^hv^h)_{\Phi}\|_{{L}^1_t(\dot{B}^{\frac{n-1}{p}-s,\frac{1}{2}}_{p,1})}\\
&\lesssim \eps\|{v}^h_{\Phi}\|_{{L}^1_t(\dot{B}^{\frac{n-1}{p},\frac{1}{2}}_{p,1})}
\|{v}^h_{\Phi}\|_{\widetilde{L}^\infty_t(\dot{B}^{\frac{n-1}{p}+1-s,\frac{1}{2}}_{p,1})}.
\endaligned$$
Similarly, the fact that $\eps^2(-\Delta_\eps)^{-1}\partial^2_n$ is a bounded operator applied for frequency localized
functions in  $L^p_h(L^2_v)$ implies
$$
\eps^2\|\partial_nq^2_{\Phi}\|_{{L}^1_t(\dot{B}^{\frac{n-1}{p}-s,\frac{1}{2}}_{p,1})}\lesssim \|\nabla_h(v^nv^h)_{\Phi}\|_{{L}^1_t(\dot{B}^{\frac{n-1}{p}-s,\frac{1}{2}}_{p,1})},
$$
$$
\eps^2\|\partial_nq^3_{\Phi}\|_{{L}^1_t(\dot{B}^{\frac{n-1}{p}-s,\frac{1}{2}}_{p,1})}\lesssim \|(v^n\mathrm{div}_hv^h)_{\Phi}\|_{{L}^1_t(\dot{B}^{\frac{n-1}{p}-s,\frac{1}{2}}_{p,1})}.
$$
Thus, applying  Lemma \ref{p2.9}, we have
$$\aligned
\eps^2\|\partial_nq^2_{\Phi}\|_{L^1_t(\dot{B}^{\frac{n-1}{p}-s,\frac{1}{2}}_{p,1})}
+\eps^2\|\partial_nq^3_{\Phi}\|_{L^1_t(\dot{B}^{\frac{n-1}{p}-s,\frac{1}{2}}_{p,1})}
\lesssim \|{v}^n_{\Phi}\|_{{L}^1_t(\dot{B}^{\frac{n-1}{p},\frac{1}{2}}_{p,1})}
\|{v}^h_{\Phi}\|_{\widetilde{L}^\infty_t(\dot{B}^{\frac{n-1}{p}+1-s,\frac{1}{2}}_{p,1})}.
\endaligned$$
Then we obtain that
$$\aligned
\eps^2\|\partial_nq_{\Phi}\|_{{L}^1_t(\dot{B}^{\frac{n-1}{p},\frac{1}{2}}_{p,1})}\lesssim Y(t)\Psi(t) +X(t)\Psi(t).
\endaligned$$
Combining all the above estimates,  we can get the bound of ${v}^n_{\Phi}$ in $\widetilde{L}^\infty_t(\dot{B}^{\frac{n-1}{p}-s,\frac{1}{2}}_{p,1})\cap L^1_t(\dot{B}^{\frac{n-1}{p},\frac{1}{2}}_{p,1})$ by the following:
\begin{align}\label{e3.38}
\begin{split}
\|{v}^n_{\Phi}\|_{\widetilde{L}_t^\infty(\dot{B}^{\frac{n-1}{p}-s,\frac{1}{2}}_{p,1})}+
\|{v}^n_{\Phi}\|_{L_t^1(\dot{B}^{\frac{n-1}{p},\frac{1}{2}}_{p,1})}
\lesssim Y_0
 +(X(t)+Y(t))\Psi(t).
\end{split}
\end{align}
This completes the proof of $Y(t)$ in Lemma \ref{l3.1}.

The following is devoted to getting the estimate of $X(t).$
The horizontal component ${v}^h$ in each dyadic block satisfies
\begin{equation}\label{e3.43}
\partial_t{\Delta}_{k,j}{v}^h+D_\eps^s{\Delta}_{k,j}{v}^h
=-{\Delta}_{k,j}\mathrm{div}_h({v}^h\otimes{v}^h)-{\Delta}_{k,j}\partial_n(v^n{v}^h)-\nabla_h{\Delta}_{k,j}q.
\end{equation}
Denote $F\overset{\mathrm{def}}=-\mathrm{div}_h({v}^h\otimes{v}^h)-\partial_n(v^n{v}^h)$, then we infer that
\begin{align}\label{e3.44}
\begin{split}
\|{\Delta}_{k,j}{v}_\Phi^h\|_{L^p_h(L^2_v)}&\lesssim e^{-c(2^{ks}+\eps^s2^{js})t}\|{\Delta}_{k,j} e^{\alpha D_n }v^h_0\|_{L^p_h(L^2_v)}\\
&\quad+\int_0^te^{-c(2^{ks}+\eps^s2^{js})(t-\tau)}\|{\Delta}_{k,j}F_{\Phi}\|_{L^p_h(L^2_v)}d\tau\\
&\quad+\int_0^te^{-c(2^{ks}+\eps^s2^{js})(t-\tau)}\|{\Delta}_{k,j}\nabla_hq_{\Phi}\|_{L^p_h(L^2_v)}d\tau.
\end{split}
\end{align}
By Young's inequality, it holds that
$$\aligned
\|\int_0^te^{-c(2^{ks}+\eps^s2^{js})(t-\tau)}\|{\Delta}_{k,j}
\partial_n(v^nv^h)_{\Phi}\|_{L^p_h(L^2_v)}d\tau\|_{L^\infty_t}\lesssim \frac{1}{\eps}\|{\Delta}_{k,j}(v^nv^h)_{\Phi}\|_{L^\frac{s}{s-1}_t(L^p_h(L^2_v))},
\endaligned$$
$$\aligned
2^{ks}\|\int_0^te^{-c(2^{ks}+\eps^s2^{js})(t-\tau)}\|{\Delta}_{k,j}
\partial_n(v^nv^h)_{\Phi}\|_{L^p_h(L^2_v)}d\tau\|_{L^1_t}\lesssim \frac{1}{\eps}2^k\|{\Delta}_{k,j}(v^nv^h)_{\Phi}\|_{L^1_t(L^p_h(L^2_v))}.
\endaligned$$
Here and in what follows, if $s=1$, the quantity $\|{\Delta}_{k,j}(v^nv^h)_{\Phi}\|_{L^\frac{s}{s-1}_t(L^p_h(L^2_v))}$ should be regarded as $\|{\Delta}_{k,j}(v^nv^h)_{\Phi}\|_{L^\infty_t(L^p_h(L^2_v))}.$
Therefore, taking $L^\infty$ norm and $L^1$ norm  on $[0,t]$,  we deduce  that
\begin{align}\label{e3.45}
\begin{split}
\|{\Delta}_{k,j}&{v}_\Phi^h\|_{L_t^\infty(L^p_h(L^2_v))}+2^{ks}
\|{\Delta}_{k,j}{v}_\Phi^h\|_{L_t^1(L^p_h(L^2_v))}\\
&\lesssim \|e^{\alpha  D_n }{\Delta}_{k,j}v^h_0\|_{L^p_h(L^2_v)}+2^k\|{\Delta}_{k,j}(v^h\otimes v^h)_{\Phi}\|_{L^1_t(L^p_h(L^2_v))}\\
&\quad+\frac{1}{\eps}2^k\|{\Delta}_{k,j}(v^n v^h)_{\Phi}\|_{L^1_t(L^p_h(L^2_v))}+\frac{1}{\eps}\|{\Delta}_{k,j}(v^n v^h)_{\Phi}\|_{L^\frac{s}{s-1}_t(L^p_h(L^2_v))}\\
&\quad+\|{\Delta}_{k,j}\nabla_hq_{\Phi}\|_{L^1_t(L^p_h(L^2_v))}.
\end{split}
\end{align}
Multiplying both sides of (\ref{e3.45}) by $2^{(\frac{n-1}{p}-s)k}2^{\frac{1}{2}j}$ and taking the sum over $k,j$, we finally get
\begin{align}\label{e3.46}
\begin{split}
\eps\|{v}_\Phi^h\|_{\widetilde{L}_t^\infty(\dot{B}^{\frac{n-1}{p}-s,\frac{1}{2}}_{p,1})}&+
\eps\|{v}_\Phi^h\|_{{L}_t^1(\dot{B}^{\frac{n-1}{p},\frac{1}{2}}_{p,1})}\\
&\lesssim\eps\|e^{\alpha D_n }v^h_0\|_{\dot{B}^{\frac{n-1}{p}-s,\frac{1}{2}}_{p,1}}
+\eps\|(v^h\otimes v^h)_{\Phi}\|_{{L}^1_t(\dot{B}^{\frac{n-1}{p}+1-s,\frac{1}{2}}_{p,1})}\\
&\quad+\|(v^n v^h)_{\Phi}\|_{\widetilde{L}^\frac{s}{s-1}_t(\dot{B}^{\frac{n-1}{p}-s,\frac{1}{2}}_{p,1})}+\|(v^n v^h)_{\Phi}\|_{{L}^1_t(\dot{B}^{\frac{n-1}{p}+1-s,\frac{1}{2}}_{p,1})}\\
&\quad+\eps\|\nabla_hq_{\Phi}\|_{{L}^1_t(\dot{B}^{\frac{n-1}{p}-s,\frac{1}{2}}_{p,1})}.
\end{split}
\end{align}

For the pressure term $q=q^1+q^2+q^3$, we find that
$$\aligned
\|\nabla_hq^1_{\Phi}\|_{L^1_t(\dot{B}^{\frac{n-1}{p}-s,\frac{1}{2}}_{p,1})}&=\|(-\Delta_\eps)^{-1}\nabla_h\partial_i\partial_j
(v^iv^j)_{\Phi}\|_{{L}^1_t(\dot{B}^{\frac{n-1}{p}-s,\frac{1}{2}}_{p,1})}\lesssim \|\nabla_h(v^hv^h)_{\Phi}\|_{{L}^1_t(\dot{B}^{\frac{n-1}{p}-s,\frac{1}{2}}_{p,1})},
\endaligned$$
where we have used that $(-\Delta_\eps)^{-1}\partial_i\partial_j$ is a bounded operator for frequency localized
functions in  $L^p_h(L^2_v)$.
Similarly,
\begin{eqnarray}\nonumber
\|\nabla_hq^2_{\Phi}\|_{L^1_t(\dot{B}^{\frac{n-1}{p}-s,\frac{1}{2}}_{p,1})} &=& 2\|(-\Delta_\eps)^{-1}\nabla_h\partial_i\partial_n
(v^iv^n)_{\Phi}\|_{{L}^1_t(\dot{B}^{\frac{n-1}{p}-s,\frac{1}{2}}_{p,1})}\\\nonumber &\lesssim& \frac{1}{\eps}\|\nabla_h(v^nv^h)_{\Phi}\|_{{L}^1_t(\dot{B}^{\frac{n-1}{p}-s,\frac{1}{2}}_{p,1})},
\end{eqnarray}
\begin{eqnarray}\nonumber
\|\nabla_hq^3_{\Phi}\|_{L^1_t(\dot{B}^{\frac{n-1}{p}-s,\frac{1}{2}}_{p,1})} &=& 2\|(-\Delta_\eps)^{-1}\nabla_h\partial_n
(v^n\mathrm{div}_hv^h)_{\Phi}\|_{{L}^1_t(\dot{B}^{\frac{n-1}{p}-s,\frac{1}{2}}_{p,1})}\\\nonumber &\lesssim& \frac{1}{\eps}\|(v^n\mathrm{div}_hv^h)_{\Phi}\|_{{L}^1_t(\dot{B}^{\frac{n-1}{p}-s,\frac{1}{2}}_{p,1})}.
\end{eqnarray}
According to Lemma \ref{p2.9}, the right hand side of (\ref{e3.46}) can be bounded by following:
$$\eps\|({v}^h\otimes{v}^h)_{\Phi}\|_{{L}^1_t(\dot{B}^{\frac{n-1}{p}+1-s,\frac{1}{2}}_{p,1})}\lesssim \eps\|{v}^h_{\Phi}\|_{{L}^1_t(\dot{B}^{\frac{n-1}{p},\frac{1}{2}}_{p,1})}
\|{v}^h_{\Phi}\|_{\widetilde{L}^\infty_t(\dot{B}^{\frac{n-1}{p}+1-s,\frac{1}{2}}_{p,1})},$$
$$\|(v^n{v}^h)_{\Phi}\|_{{L}^1_t(\dot{B}^{\frac{n-1}{p}+1-s,\frac{1}{2}}_{p,1})}\lesssim \|{v}^h_{\Phi}\|_{\widetilde{L}^\infty_t(\dot{B}^{\frac{n-1}{p}+1-s,\frac{1}{2}}_{p,1})}
\|{v}^n_{\Phi}\|_{{L}^1_t(\dot{B}^{\frac{n-1}{p},\frac{1}{2}}_{p,1})},$$
$$\aligned
\eps\|\nabla_hq_{\Phi}\|_{L^1_t(\dot{B}^{\frac{n-1}{p}-s,\frac{1}{2}}_{p,1})}&\lesssim \Big(\eps\|{v}^h_{\Phi}\|_{{L}^1_t(\dot{B}^{\frac{n-1}{p},\frac{1}{2}}_{p,1})}+\|{v}^n_{\Phi}\|_{{L}^1_t(\dot{B}^{\frac{n-1}{p},\frac{1}{2}}_{p,1})}\Big)
\|{v}^h_{\Phi}\|_{\widetilde{L}^\infty_t(\dot{B}^{\frac{n-1}{p}+1-s,\frac{1}{2}}_{p,1})},\\
\endaligned$$
$$\aligned\|(v^n{v}^h)_{\Phi}\|_{\widetilde{L}^\frac{s}{s-1}_t(\dot{B}^{\frac{n-1}{p}-s,\frac{1}{2}}_{p,1})}&\lesssim \|{v}^n_{\Phi}\|_{\widetilde{L}^\frac{s}{s-1}_t(\dot{B}^{\frac{n-1}{p}-1,\frac{1}{2}}_{p,1})}
\|{v}^h_{\Phi}\|_{\widetilde{L}^\infty_t(\dot{B}^{\frac{n-1}{p}+1-s,\frac{1}{2}}_{p,1})}\\
&\lesssim \Big(\|{v}^n_{\Phi}\|_{\widetilde{L}^\infty_t(\dot{B}^{\frac{n-1}{p}-s,\frac{1}{2}}_{p,1})}+\|{v}^n_{\Phi}\|_{{L}^1_t(\dot{B}^{\frac{n-1}{p},\frac{1}{2}}_{p,1})}\Big)
\|{v}^h_{\Phi}\|_{\widetilde{L}^\infty_t(\dot{B}^{\frac{n-1}{p}+1-s,\frac{1}{2}}_{p,1})}.
\endaligned$$
These imply that
\begin{align}\label{e3.47}
\begin{split}
\eps\|{v}_\Phi^h\|_{\widetilde{L}_t^\infty(\dot{B}^{\frac{n-1}{p}-s,\frac{1}{2}}_{p,1})}&+\eps
\|{v}_\Phi^h\|_{L_t^1(\dot{B}^{\frac{n-1}{p},\frac{1}{2}}_{p,1})}\lesssim X_0 +(X(t)+Y(t))\Psi(t).
\end{split}
\end{align}

Combining (\ref{e3.47}) with (\ref{e3.38}), we finally obtain that there exists a constant $C_2$ such that
\begin{align}\label{e3.50}
\begin{split}
X(t)+Y(t)\leq C_2 (X_0+Y_0) +C_2(X(t)+Y(t))\Psi(t).
\end{split}
\end{align}
This completes the proof of Lemma \ref{l3.1}.
\end{proof}
\section{Proof of the main result}
In this section, we will prove the Theorem \ref{t1.2}. It relies on a continuation argument.
For any $\lambda>\lambda_0$ and  $\eta_1$, we
define $\tau$ by
\begin{equation}\label{e3.51}
\tau=\max\{t\geq0|\ X(t)+Y(t)\leq \eta_1,\quad \ \Psi(t)\leq \eta_1\}.
\end{equation}
In what follows, we shall prove that $\tau=\infty$ under the assumption (\ref{e1.3}) for some small number $\eta_1$. Assume that this is not true.
We choose $\eta_1$ small enough such that
$$\theta(\tau)\leq C Y(\tau)\leq C \eta_1 \leq \frac{\alpha}{2\lambda},\quad (C_1 +C_2)\eta_1\leq\frac{1}{4}.$$
For such fixed $\eta_1$, we select the following  norms of  initial data sufficiently small enough such that
$$
C_1\Psi(0)+C_2(X(0)+Y(0))\leq C\eta \leq \frac{\eta_1}{4}.$$
Hence, we obtain from
 Lemma \ref{l3.2}  and  \ref{l3.1} that
\begin{align}\label{e3.52}
\begin{split}
&\Psi(\tau)\leq C_1\Psi(0)+C_1\eta_1^2,\quad X(\tau)+Y(\tau)\leq C_2(X(0)+Y(0))+C_2\eta_1^2.
\end{split}
\end{align}
This implies that
\begin{align}\label{e3.53}
\begin{split}
X(\tau)+Y(\tau)\leq \frac{\eta_1}{2},\ \ \Psi(\tau)\leq \frac{\eta_1}{2}.
\end{split}
\end{align}
However, this contradicts  (\ref{e3.51}) and hence completes the proof.

\section*{Acknowledgement}
The authors were in part
supported by NSFC (grants No. 11171072, 11421061 and 11222107), Shanghai Talent Development Fund and
SGST 09DZ2272900.


\begin{thebibliography}{999}

\bibitem{BCD} H. Bahouri, J. Y.  Chemin and R. Danchin,
 \textit{Fourier analysis and nonlinear partial differential equations},
  Grundlehren der mathematischen Wissenschaften, 343(2011).

\bibitem{Bour} J. Bourgain, N. Pavlovic.  \textit{Ill-posedness of the Navier-Stokes equations in a critical space in 3D}, Journal of Functional Analysis, 255, 2233-2247(2008).

\bibitem{Bo} J. M. Bony, \textit{Calcul symbolique et propagation des singularit$\acute{e}$s pour $\acute{e}$quations
aux d$\acute{e}$riv$\acute{e}$es partielles nonlin$\acute{e}$aires},
Annales Scinentifiques de l'$\acute{e}$cole Normale
Sup$\acute{e}$rieure 14, 209-246(1981).


\bibitem{Ch} J. Y. Chemin,
\textit{Le syst\`{e}me de Navier-Stokes incopressible soixante dix ans apr\`{e}s Jean Leray},  S\'{e}minaire et congr\`{e}s, 9, 99-123(2004).

\bibitem{Ch1} J. Y. Chemin,
\textit{Localization in Fourier space and Navier-Stokes system,
Phase Space Analysis of Partial Differential Equations}, Proceedings
2004, CRM series, Pisa, 53-136(2004).


\bibitem{Ch2} J. Y. Chemin,
\textit{theor\'{e}m\'{e}s d'unicit\'{e} pour le syst\'{e}me de navier-Stokes tridimensionnel}, J. Anal. Math, 77, 27-50(2009).


\bibitem{CG} J. Chemin and I. Gallagher,
\textit{Large, global solutions to the Navier-Stokes euqations, slowly varying in one direction},
 Transactions of the American Mathematical Society, 362, 173, 983-1012(2011).



\bibitem{CGP} J. Chemin, I. Gallagher and M. Paicu
\textit{Global regularity for some classes of large solutions to the Navier-Stokes equations},
 Annals of Mathematics 2, 173, 983-1012(2011).




\bibitem{CL} J. Y. Chemin and N. Lerner,
 \textit{Flot de champs de vecteurs non Lipschitziens et $\acute{e}$quations
de Navier-Stokes},
J. Differential Equations,  121, 314-328(1995).

\bibitem{CMP}M. Cannone, Y. Meyer and F. Planchon,
\textit{Solutions autosimilaries des $\acute{e}$quations de Navier-Stokes},
S$\acute{e}$minaire $\acute{E}$quations aux D$\acute{e}$riv$\acute{e}$es Partielles de l'$\acute{E}$cole Polytechnique, 1993-1994.

\bibitem{Da3}R. Danchin,
 \textit{Local and global well-posedness results for flows of inhomogeneous viscous fluids},
Advances in  Differential Equations, 9, 353-386(2004).

\bibitem{DH}D. Fang and B. Han,
 \textit{Global solution for the generalized anisotropic Navier-Stokes equations with Large data},
to appear in Mathematical Modelling and Analysis, 2015.

\bibitem{FK} H. Fujita and T. Kato,
\textit{On the Navier-Stokes initial value problem I},
Archive for Rational Mechanics and Analysis, 16, 269-315(1964).

\bibitem{GHZ} G. Gui, J. Huang and P. Zhang,
\textit{Large global solutions to 3-D inhomogeneous Navier-Stokes equations slowly varying in one variable},
Journal of Functional Analysis,  261, 3181-3210(2011).

\bibitem{Ha} B. Han,
 \textit{Global reqularity to the 3D incompressible Navier-Stokes equations with large initial data},
submitted to Journal of Functional Analysis, 2014.

\bibitem{TLL} Y. Thomas Hou, Z. Lei and C. Li,
\textit{Global regularity of the 3D axi-symmetric Navier-Stokes equations with
anisotropic data},
Comm. Partial Differential Equations, 33, 1622-1637(2008).

\bibitem{If} D. Iftimie,
\textit{The resolution of the Navier-Stokes equations in anisotropic spaces},
 Revista Matematica Ibero-Americana. 15, 1-36(1999).


 \bibitem{IRS} D. Iftimie, G. Raugel and G. R. Sell,
\textit{Navier-Stokes equations in thin  3D domains with the Navier boundary conditions},
Indiana University Mathematical Journal, 56, 1083-1156(2007).



 \bibitem{Ka} T. Kato,
\textit{Strong $L^p$ solutions of the Navier-Stokes equations in $\Real^m$ with applications to weak solutions},
 Mathematische Zeitschrift, 187, 471-480(1984).


 \bibitem{KT} H. Koch and D. Tataru,
\textit{Well-posedness for the Navier-Stokes equations},
Advances in Maththematics, 157, 22-35(2001).

\bibitem{LL}  Z. Lei and F. Lin
\textit{Global mild solutions of Navier-Stokes equations},
 Comm. Pure Appl. Math. , 64, 1297-1304(2011).




\bibitem{LLZ} Z. Lei, F. Lin and Y. Zhou
\textit{Structure of Helicity and global solutions of incompressible Navier-Stokes equation},
preprint, 2014.

\bibitem{Le} J. Leray,
\textit{Essai sur le mouvement d'un liquide visqueux emplissant l'espace},
Acta Mathematica, 63, 193-248(1933).

\bibitem{MN} A. Mahalov and B. Nicolaenko,
\textit{Global solvability of three dimensional Navier-Stokes equations with Uniformly high initial vorticity},
(Russian. Russian summary) Uspekhi Mat. Nauk 58, 79-110(2003),  Translation in Russian Math. Surveys, 287-318 (2003).

\bibitem{PZ} M. Paicu and Z. Zhang,
\textit{Global regularity for the Navier-Stokes equations with large, slowly varying initial data in the vertical direction},
Analysis of Partial Differential equation, 4, 95-113(2011).

\bibitem{PZ2} M. Paicu and Z.  Zhang,
\textit{Global well-posedness for 3D Navier-Stokes equationswith ill-prepared initial data},
J. Inst. Math. Jussieu, 13, 395-411(2014).

\bibitem{RS} G. Raugel and G. R. Sell,
\textit{Navier-Stokes equations on thin  3D domains. I. Global attractors and global regularity of solutions},
Journal of the American Mathematical Society, 6, 503-568(1993).



\end{thebibliography}
\end{document}